\documentclass[11pt]{amsart}
\usepackage[margin=1in]{geometry} 
  
\usepackage{setspace}  
\usepackage{amsmath, amsfonts, amsthm, amstext, xspace}
\usepackage{comment}
\usepackage{amssymb,latexsym}
\usepackage{MnSymbol}
\usepackage{scalerel}
\usepackage{dsfont} 
\usepackage{sseq}
\usepackage{tikz-cd}
\usepackage[all]{xy} 
\usepackage{graphicx}
\theoremstyle{definition}
\newtheorem{thm}{Theorem}[section]

\newtheorem{meta-thm}{Meta-Theorem}[thm]
\newtheorem{lem}[thm]{Lemma}
\newtheorem{cor}[thm]{Corollary}

\newtheorem{prop}[thm]{Proposition}

\newtheorem{rem}[thm]{Remark}
\newtheorem{exm}[thm]{Example}

\newtheorem{definition-proposition}[thm]{Definition-Proposition}
\newtheorem{running notations}[thm]{Running notations}
\newtheorem{notation}[thm]{Notation}
\newtheorem{remark}[thm]{Remark}

\newtheorem{convention}[thm]{Conventions}
\newtheorem{construction}[thm]{Construction}

\newtheorem*{sketch of proof}{Sketch of proof}
\theoremstyle{definition}
\newtheorem{defin}[thm]{Definition}

\newtheorem*{thm*}{Theorem}

\def\co{\colon\thinspace}

\usepackage{graphicx}
\graphicspath{{./EU-symbol.jpg}}

\DeclareMathOperator{\bfk}{\bf k}

\DeclareMathOperator{\bL}{\mathbb{L}}

\DeclareMathOperator{\bZ}{\mathbb{Z}}

\DeclareMathOperator{\bT}{\mathbb{T}}

\DeclareMathOperator{\cA}{\mathcal{A}}

\DeclareMathOperator{\cF}{\mathcal{F}}
\DeclareMathOperator{\cI}{\mathcal{I}}

\DeclareMathOperator{\cR}{\mathcal{R}}
\DeclareMathOperator{\cU}{\mathcal{U}}
\DeclareMathOperator{\cV}{\mathcal{V}}
\DeclareMathOperator{\cW}{\mathcal{W}}

\newcommand{\m}[1]{\underline{#1}}
\newcommand{\mM}{\m{M}}
\newcommand{\mZ}{\m{\bZ}}

\newcommand{\mA}{\m{A}}
\newcommand{\mW}{\m{\mathbb{W}}}
\newcommand{\mN}{\m{N}}

\newcommand{\upi}{\m{\pi}}

\DeclareMathOperator{\HH}{HH}
\DeclareMathOperator{\W}{\operatorname{W}}

\DeclareMathOperator{\coker}{coker}

\DeclareMathOperator{\id}{\operatorname{id}}

\DeclareMathOperator{\THH}{\operatorname{THH}}
\DeclareMathOperator{\TC}{\operatorname{TC}}
\DeclareMathOperator{\HR}{\underline{\operatorname{HR}}}

\DeclareMathOperator{\THR}{\operatorname{THR}}

\DeclareMathOperator{\TRR}{\operatorname{TRR}}
\DeclareMathOperator{\TCR}{\operatorname{TCR}}
\DeclareMathOperator{\K}{\operatorname{K}}
\DeclareMathOperator{\KR}{\operatorname{KR}}

\DeclareMathOperator{\TR}{\operatorname{TR}}

\DeclareMathOperator{\Comm}{\operatorname{Comm}}
\DeclareMathOperator{\All}{\mathcal{A}\ell\ell}

\DeclareMathOperator{\tr}{\operatorname{tr}}
\DeclareMathOperator{\di}{\textup{di}}
\DeclareMathOperator{\op}{\textup{op}}
\DeclareMathOperator{\res}{\textup{res}}

\DeclareMathOperator{\cy}{\textup{cy}}

\DeclareMathOperator{\Map}{\operatorname{Map}}
\DeclareMathOperator{\Stab}{\operatorname{Stab}}

\DeclareMathOperator{\holim}{\operatorname{holim}\thinspace}
\DeclareMathOperator{\Tor}{\operatorname{Tor}}

\DeclareMathOperator{\sq}{\operatorname{sq}}

\newcommand{\msf}[1]{\mathsf{#1}}

\DeclareMathOperator{\Mak}{\textup{Mack}}
\DeclareMathOperator{\Tamb}{\msf{Tamb}}

\DeclareMathOperator{\Assoc}{\msf{Assoc}}
\DeclareMathOperator{\Alg}{\msf{Alg}}

\DeclareMathOperator{\Sp}{\msf{Sp}}

\DeclareMathOperator{\Set}{\msf{Set}}

\DeclareMathOperator{\Top}{\msf{Top}}

\DeclareMathOperator{\ho}{\msf{Ho}}
\DeclareMathOperator{\Ind}{Ind}

\makeatletter
\DeclareRobustCommand\bigop[1]{%
  \mathop{\vphantom{\sum}\mathpalette\bigop@{#1}}\slimits@
}
\newcommand{\bigop@}[2]{%
  \vcenter{%
    \sbox\z@{$#1\sum$}%
    \hbox{\resizebox{\ifx#1\displaystyle1.1\fi\dimexpr\ht\z@+\dp\z@}{!}{$\m@th#2$}}%
  }%
}
\makeatother

\newcommand{\bigbox}{\DOTSB\bigop{\square}}
\newcommand{\timesover}[1]{\underset{#1}{\times}}

 \usepackage[
 textwidth=3cm,
 textsize=small,
 colorinlistoftodos]
{todonotes}

\usepackage{hyperref}
\hypersetup{colorlinks=true, linkcolor=blue, menucolor=blue}
\usepackage{cleveref}

\crefdefaultlabelformat{#2(#1)#3}

\title[THR via the norm and Real Witt vectors]{Real topological Hochschild homology via the norm and Real Witt vectors} 
\author{Gabriel Angelini-Knoll}
\address{Department of Mathematics, Applied Mathematics and Statistics, Case Western Reserve University, Cleveland, OH, USA}
\email{gja39@case.edu}
\author{Teena Gerhardt}\address{Department of Mathematics, Michigan State University, East Lansing, MI 48824, USA}\email{teena@math.msu.edu}
\author{Michael A. Hill}\address{School of Mathematics, University of Minnesota, Minneapolis, Minnesota 55455}\email{mahill@umn.edu}
\date{}							

\begin{document} 
\begin{abstract}
We prove that Real topological Hochschild homology can be characterized as the norm from the cyclic group of order $2$ to the orthogonal group $O(2)$. From this perspective, we then prove a multiplicative double coset formula for the restriction of this norm to dihedral groups of order $2m$. This informs our new definition of Real Hochschild homology of rings with anti-involution, which we show is the algebraic analogue of Real topological Hochschild homology. Using extra structure on Real Hochschild homology, we define a new theory of $p$-typical Witt vectors of rings with anti-involution. We end with an explicit computation of the degree zero $D_{2m}$-Mackey functor homotopy groups of $\mathrm{THR}(\underline{\mathbb{Z}})$ for $m$ odd. This uses a Tambara reciprocity formula for sums for general finite groups, which may be of independent interest.
\end{abstract}

\maketitle

\section{Introduction}

Real algebraic $\K$-theory, $\KR$, is an invariant of a ring (spectrum)  
with anti-involution~\cite{HM15}. 
It is a generalization of Karoubi's Hermitian $\K$-theory~\cite{Karoubi}, and an analogue of Atiyah's  topological $\K$-theory with reality~\cite{Atiyah}. Real topological Hochschild homology, $\THR(A)$, is an $O(2)$-equivariant spectrum that receives a trace map from Real algebraic $\K$-theory~\cite{HM15, Dotto12, Hog16}. 
Just as topological Hochschild homology is essential to the trace method approach to algebraic $\K$-theory, $\THR$ is useful for computing Real algebraic $\K$-theory. 

In Hill, Hopkins, and Ravenel's solution to the Kervaire invariant one problem~\cite{HHR}, they developed the theory of a multiplicative norm functor $N_H^G$ from $H$-spectra to $G$-spectra, for $H \le G$ finite groups. 
In the case of non-finite compact Lie groups, however, norms are currently only accessible in a few specific cases. 
In~\cite{ABGHLM18}, the authors extend the norm construction to the circle group $\bT$, defining the equivariant norm $N_e^{\bT}(R)$ for an associative ring spectrum $R$. They further show that this equivariant norm is a model for topological Hochschild homology. In the present paper, we extend the norm construction to an equivariant norm from $D_2$ to $O(2)$, using the dihedral bar construction, demonstrating that this norm is a model for Real topological Hochschild homology. 

Before defining the norm, we introduce some notation. We write $B^{\di}_{\bullet}(A)$ for the dihedral bar construction on $A$ (cf. Section~\ref{sec: norm}). 
Let $\cU$ denote a complete $O(2)$-universe, and let
$\cV$ be the complete $D_{2}$-universe constructed by restricting $\cU$ to $D_{2}$. We write $\widetilde{\cV}$ for the $O(2)$-universe associated to $\mathcal{V}$ by inflation along the determinant homomorphism. The input for Real topological Hochschild homology is a ring spectrum with anti-involution, $(A,\omega)$. Ring spectra with anti-involution can be alternatively described as algebras in $D_2$-spectra over an $E_{\sigma}$-operad, where $\sigma$ is the sign representation of $D_2$, the cyclic group of order 2.  We refer the reader to~\cite{Hil17} for more details. 
\begin{defin}\label{O2 norm}
We define the functor 
\[N_{D_2}^{O(2)} \colon \thinspace \Assoc_{\sigma}(\Sp_{\cV}^{D_2})\longrightarrow \Sp _{\cU}^{O(2)} \]
as the composite 
$\cI_{\widetilde{\cV}}^{\cU}| B_{\bullet}^{\di} (-)  |.
$ 
Here $\cI_{\widetilde{\cV}}^{\cU}$ denotes the change of universe functor.
\end{defin} 
We then prove that this functor satisfies 
one of the fundamental properties of equivariant norms: in the  commutative setting it is left adjoint to restriction. Let $\cR$ denote the family of subgroups of $O(2)$ which intersect $\bT$ trivially, and let $\Sp ^{O(2),\cR}_{\cU}$ denote the $\cR$-model structure on genuine $O(2)$-spectra~\cite[Appendix B]{HHR} (cf.~\cite[Theorems~2.26, 2.29]{ABGHLM18}).  
 \begin{thm}
 \label{main thm 1}
 The restriction  
 \[ 
 N_{D_2}^{O(2)} \co \Comm(\Sp_{\cV}^{D_2})\to \Comm(\Sp ^{O(2),\cR}_{\cU})
 \]
 of the norm functor $N_{D_2}^{O(2)}$ to genuine commutative $D_2$-ring spectra is left Quillen adjoint to the restriction functor $\iota_{D_2}^*$.
 \end{thm} 
This equivariant norm, $N_{D_2}^{O(2)}A$, is a model for Real topological Hochschild homology, $\THR(A)$. 
 
 In~\cite{DMPPR21}, the authors prove that for a flat ring spectrum  with anti-involution $(R, \omega)$ in the sense of~\cite[Definition~2.6]{DMPPR21}, there is a stable equivalence of $D_2$-spectra
 \[
 \iota^*_{D_2}\THR(R) \simeq R \wedge^{\bL}_{N_e^{D_2} \iota_e^* R} R.
 \]
We generalize this result by proving a multiplicative double coset formula for the norm $N_{D_2}^{O(2)}$.   
\begin{thm}[Multiplicative Double Coset Formula]\label{Introthm:doublecoset}
Let $\zeta$ denote the $2m$-th root of unity $e^{2\pi i/2m}$. 
When $R$ is a flat $E_{\sigma}$-ring and $m$ is a positive integer, there is a stable equivalence of $D_{2m}$-spectra 
\[ \iota_{D_{2m}}^*N_{D_2}^{O(2)}R \simeq  N_{D_2}^{D_{2m}}R \wedge^{\bL}_{N_e^{D_{2m}}\iota_e^*R} N_{\zeta D_2\zeta^{-1}}^{D_{2m}}c_{\zeta}R. \]
\end{thm}

Ordinary topological Hochschild homology is the topological analogue of classical  
Hochschild homology of algebras.  
The two theories are related by a linearization map.
\[
\pi_k\THH(R) \to \HH_k(\pi_0R),
\]
which is an isomorphism in degree 0, where $R$ is a $(-1)$-connected ring spectrum. It is natural to ask, then, what is the algebraic analogue of Real topological Hochschild homology? In this paper, we define such an analogue: a theory of Real Hochschild homology for rings with anti-involution, or more generally for discrete $E_{\sigma}$-rings. 
A discrete $E_{\sigma}$-ring is a type of $D_2$-Mackey functor that arises as the algebraic analogue of $E_{\sigma}$-rings in spectra (cf. Definition~\ref{defn:discreteEsigma}). Indeed, if $R$ is an $E_{\sigma}$-ring in spectra, $\m{\pi}_0^{D_2}(R)$ is a discrete $E_{\sigma}$-ring. 
\begin{defin}
The \emph{Real $D_{2m}$-Hochschild homology} of a discrete $E_{\sigma}$-ring $\underline{M}$ is the graded $D_{2m}$-Mackey functor
\[ \HR_*^{D_{2m}}(\m{M}) = H_*\left ( \HR_{\bullet}^{D_{2m}}(\m{M}) \right ). \]
where 
\[ \HR_{\bullet}^{D_{2m}}(\m{M}) =B_{\bullet}(N_{D_2}^{D_{2m}}\m{M},N_{e}^{D_{2m}}\iota_e^*\m{M}, N_{\zeta D_2 \zeta^{-1}}^{D_{2m}}c_{\zeta}\m{M} ).\]
\end{defin}
\noindent This theory of Real Hochschild homology is computable using homological algebra for Mackey functors. We prove that Real topological Hochschild homology and Real Hochschild homology are related by a linearization map.

 \begin{prop}
 For any $(-1)$-connected  $E_\sigma$-ring $A$, we have a natural homomorphism
 \[
 \m{\pi}_k^{D_{2m}} \THR(A) \longrightarrow  \HR_k^{D_{2m}}(\m{\pi}_0^{D_2}A). 
 \]
 which is an isomorphism when $k=0$.
 \end{prop}
\noindent This relationship facilitates the computation of Real topological Hochschild homology. As an example, we compute the degree zero $D_{2m}$-Mackey functor homotopy groups of $\THR(H\mZ)$, for odd $m$. When restricted to $\underline{\pi}_0^{D_{2}}(\THR(H\mZ)^{D_{2p^k}}),$ 
this computation recovers a computation of~\cite{DMP22}, proven by different methods.

As part of this computation, we do some of the first calculations to appear in the literature of dihedral norms for Mackey functors. In doing so, we establish a Tambara reciprocity formula for sums for general finite groups.
\begin{thm}[Tambara reciprocity for sums]
     Let \(G\) be a finite group and \(H\) a subgroup, and let \(\m{R}\) be a \(G\)-Tambara functor. For each \(F\in\Map^H\big(G,\{a,b\}\big)\), let \(K_F\) be the stabilizer of \(F\). Then for any \(a,b\in\m{R}(G/H)\), we have

\[ N_H^G(a+b)= \sum_{[F]\in\Map^H(G,\{a,b\})/G} tr_{K_F}^G\left(
    \prod_{[\gamma]\in K_F\backslash G\slash H}\!\! N_{K_F\cap (\gamma^{-1}H\gamma)}^{K_F}\Big(\gamma res_{H\cap (\gamma K\gamma^{-1})}^{H}\big(F(\gamma^{-1})\big)\Big)
    \right)
    \]
\end{thm}
In the case of dihedral groups, this leads to a very explicit formula for Tambara reciprocity for sums (cf. Lemma~\ref{lem:reciprocity}), facilitating the computation of dihedral norms.

Our definition of Real Hochschild homology also 
leads to a definition of Witt vectors for rings with anti-involution. Classically, Witt vectors are closely related to topological Hochschild homology. Indeed, in~\cite{HM97} Hesselholt and Madsen show that for a commutative ring $R$
\[
\pi_0(\THH(R)^{C_{p^n}}) \cong \W_{n+1}(R;p),
\]
 where $\W_{n+1}(R;p)$ denotes the length $n+1$ $p$-typical Witt vectors of $R$. Further, in~\cite{Hes97}, Hesselholt generalized the theory of Witt vectors to non-commutative rings and showed that for any associative ring $R$ there is a relationship between Witt vectors and topological cyclic homology,
\[
\TC_{-1}(R;p) \cong \W(R;p)_F.
\]
Here $W(R)_F$ denotes the coinvariants of the Frobenius endomorphism on the $p$-typical Witt vectors $W(R;p)$.

In our work, we prove analogous results for rings with anti-involution. We provide a Real cyclotomic structure on Real Hochschild homology, and use this to define Witt vectors of rings with anti-involution $R$, denoted $\mW(R;p)$, (cf, Definition~\ref{def: Witt vectors}).  As a consequence of this work, we show that $\upi_{-1}\TCR(R)$ can be described purely in terms of equivariant homological algebra. 
 \begin{thm}
 Let $A$ be an $E_{\sigma}$-ring and assume $R^1\lim_{k}\pi_1^{C_2}\THR(A)^{\mu_{p^k}}=0$.
 There is an isomorphism
\[ \upi_{-1}\TCR(A;p)\cong \mW(\upi_0^{D_2}A;p)_{F}.\]
 where $\mW(A;p)_{F}$ is the coinvariants of an  operator 
 \[ 
     F\colon \thinspace \mW(\upi_0^{D_2}A;p)\longrightarrow  \mW(\upi_0^{D_2}A;p).
 \]
 \end{thm}

\subsection{Conventions}
Let $\Top$ denote the category of based compactly generated weak Hausdorff spaces. We refer to objects in $\Top$ as spaces and morphisms in $\Top$ as maps of spaces. Let $G$ be a compact Lie group. Then $\Top^{G}$ denotes the category of based $G$-spaces and based $G$-equivariant maps of spaces. For a $G$-universe $\cU$, let $\Sp_{\cU}^{G}$ be the category of orthogonal $G$-spectra indexed on $\cU$~\cite[II.~2.6]{MM02}.  

\subsection{Acknowledgements}
The first author would like to thank V. Boelens, E. Dotto, I. Patchkoria, and H. Reich for helpful conversations. The second author would like to thank C. Lewis for helpful conversations. 
The second author was supported by NSF grants DMS-1810575, DMS-2104233, and DMS-2052042. The third author was supported by NSF grants DMS-2105019 and DMS-2052702. Some of this work was done while the second author was in residence at the Mathematical Sciences Research Institute in Berkeley, CA (supported by the National
Science Foundation under grant DMS-1440140) during the Spring 2020 semester. Angelini-Knoll is grateful to Max Planck Institute for Mathematics in Bonn for its hospitality and financial support.
This project has received funding from the European Union's Horizon 2020 research and innovation programme under the Marie Sk\l{}odowska-Curie grant agreement No 1010342555.
\thinspace \includegraphics[scale=0.1]{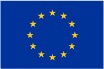}

\section{Real topological Hochschild homology via the norm}\label{sec: norm}
 Recall that the third author with Hopkins and Ravenel~\cite{HHR} define multiplicative norm functors from $H$-spectra to $G$-spectra
 \[
  N_H^G: \Sp^H \to \Sp^G
 \]
 for a finite group $G$ and subgroup $H$. It is shown in~\cite{ABGHLM18} and~\cite{BDS18} that one can extend this construction to a norm to the circle group, $N_e^{\bT}$, on associative ring orthogonal spectra. The functor 
 \[
 N_e^{\bT}\co \Assoc(\Sp) \to \Sp^{\bT}_U
 \]
 is defined via the cyclic bar construction, $N_e^{\bT}(R) = \mathcal{I}_{\mathbb{R}^{\infty}}^U|B^{cyc}_{\wedge}R|.$ It follows from this construction of $N_{e}^{\bT}$ that topological Hochschild homology can be viewed as a norm from the trivial group to $\bT$. 

In this section, 
we consider the analogous story for Real topological Hochschild homology. In particular, 
we define a norm functor $N_{D_2}^{O(2)}$ using the dihedral bar construction and 
we prove that the norm satisfies the adjointnesss properties that 
one would expect from an equivariant norm~\cite{HHR}. As Real topological Hochschild homology can be constructed via the dihedral bar construction~\cite{DMPPR21}, this gives a characterization of $\THR(A)$ as an equivariant norm $N_{D_2}^{O(2)}(A).$ 

We begin by fixing some conventions. Let $O(2)$ denote the compact Lie group of two-by-two orthogonal matrices. The determinant map determines an extension 
 \[ 1\longrightarrow \bT \longrightarrow O(2) \overset{\operatorname{det}}{\longrightarrow} \{-1,1\} \longrightarrow  1\]
 of groups. We choose a splitting by sending $-1$ to the matrix $\tau :=  ( \tiny\begin{array}{cc} 0& 1 \\ 1 & 0\normalsize \end{array} )$ and write $D_2$ for the subgroup of $O(2)$ generated by $\tau$. We then write $\mu_m\subset \bT$ for the subgroup of $m$-th roots of unity generated by $\zeta_m=e^{2\pi i/m}$. 
 Finally, we fix a presentation
 $D_{2m}=\langle \tau, \zeta_m \mid  \tau^2=\zeta_m^m=(\tau\zeta_m)^2 = 1\rangle$
 for the dihedral group of order $D_{2m}$ regarded as a subgroup of $\bT\ltimes D_2=O(2)$.

The input for Real topological Hochschild homology is a ring spectrum with anti-involution. These can alternatively be described as $E_{\sigma}$-rings, algebras in $D_2$-spectra over an $E_{\sigma}$-operad, where $\sigma$ is the sign representation (cf.~\cite{Hil17} for more details on the theory of $E_{\sigma}$-operads). For convenience, we now unpack the definition of an $E_{\sigma}$-ring. 

\begin{defin}
By an $E_1$-ring $A$, we mean an algebra over the associative operad $\Assoc$ in the category of orthogonal spectra $\Sp$. By an $E_0$-$A$-algebra we mean an $A$-bimodule $M$ equipped with a map $A\to M$ of $A$-bimodules (cf.~\cite[Rem. 2.1.3.10]{HA}). 
\end{defin}

 \begin{cor} 
 An $E_{\sigma}$-ring $R$ is exactly a $D_2$-spectrum $R$ such that 
 \begin{enumerate}
 \item the spectrum $\iota_e^*R$ is an $E_1$-ring with anti-involution, denoted  $\tau \colon \thinspace \iota_e^*R^{\textup{op}}\longrightarrow \iota_e^*R$, given by the action of the generator of the Weyl group,
 \item the spectrum $R$ is an $E_0$-$N^{D_2}_e \iota_e^*R$-algebra and applying $\iota_e^*$ to the $E_0$-$N^{D_2}_e \iota_e^*R$-algebra structure map gives $\iota_e^*R$ the standard $E_0$-$\iota_e^*R\wedge \iota_e^*R^{\op}$-algebra structure.
 \end{enumerate}
 \end{cor}

 \begin{exm}\label{exm: involution implies sigma}
  Given an $E_1$-ring with anti-involution, regarding $R$ as an object in $\Sp_{\cV}^{D_2}$ produces an $E_\sigma$-ring structure on $R$. 
 \end{exm}
 \begin{exm}\label{exm: commutative implies sigma}
 If $R$ is in $\Comm(\Sp_{\cV}^{D_2})$, then $R$ is an $E_\sigma$-ring. 
 \end{exm}

 For an $E_{\sigma}$-ring $A$, we now recall the definition of the dihedral bar construction on $A$. 

 \begin{defin}
 For $A$ an $E_{\sigma}$-ring, the dihedral bar construction $B^{\di}_{\bullet}(A)$ has $k$-simplices
 \[
 B^{\di}_{k}(A) = A \wedge A^{\wedge k}.
 \]
 It has the same simplicial structure maps and cyclic operator as the cyclic bar construction, $B^{\cy}_{\bullet}(A)$, with the addition of a levelwise involution $\omega_{k}$ acting on the $k$-simplices. To specify the levelwise involution, let $\bfk$ be the $D_2$-set $\{1,2,\ldots ,k\}$ with the generator $\alpha$ of $D_2$ acting by $\alpha(i)=k-i+1$ for $1\le i \le k$. Then we define the action of $\omega_{k}$ by the composite 
 \[ 
 	\xymatrix{ 
 		\omega_{k} \co A \wedge A^{\wedge\bfk} \ar[rr]^{\id\wedge A^{\alpha} } &&A \wedge A^{\wedge \bfk}\ar[rr]^{ \tau \wedge \tau \wedge \ldots \wedge \tau} && M \wedge A^{\wedge \bfk}
 		}
 \] 
 where $A^{\alpha}$ is the automorphism of $A^{\bfk}$ induced by $\alpha\co \bfk\to \bfk$. 
 It is straightforward to check that the structure maps satisfy the usual simplicial and cyclic identities, together with the additional relations
 \begin{align*}
 d_i\omega_{k}=&\omega_{k-1}d_{k-i} \text{ for } 0\le i \le k \\
  s_i\omega_{k}=&\omega_{k+1}s_{k-i}\text{ for } 0\le i \le k \\
 \omega_{k}t_{k}=&t_{k}^{-1}\omega_{k} \hspace{.5cm} 
 \end{align*}
 Therefore, the dihedral bar construction is a dihedral object in the sense of~\cite{Loday98}, and hence by~\cite[Theorem~5.3]{FL91} its geometric realization has an action of $O(2)$.
 \end{defin}

\begin{lem}\label{lem:O(2)}
There is a functor 
$|B_{\bullet}^{\di}(-)|\colon \thinspace  \Assoc_\sigma(\Sp_{\cV}^{D_2})\to \Sp_{\widetilde{\cV}}^{O(2)}.$
\end{lem}
\begin{proof}
By~\cite{FL91} 
the realization of a dihedral space has an $O(2)$-action. Since geometric realization in orthogonal spectra is computed level-wise, we know $|B_{\bullet}^{\di}(A)|$ has an $O(2)$-action and this orthogonal $O(2)$-spectrum is indexed on $\widetilde{\cV}$. This construction 
is also clearly functorial. 
\end{proof}

Real topological Hochschild homology can be defined via the dihedral bar construction, as in~\cite{DMPPR21, DMP21}.

\begin{defin}
For an $E_{\sigma}$-ring $A$, the Real topological Hochschild homology of $A$ is
\[
\THR(A) = |B_{\bullet}^{\di}(A)|.
\]
\end{defin}
\noindent By Lemma~\ref{lem:O(2)}, this is an $O(2)$-spectrum. We now show that it can be viewed as an equivariant norm.

 We define the norm from $D_2$ to $O(2)$ using the dihedral bar construction. 
 \begin{defin}\label{O2 norm} 
 Let $A$ be an $E_{\sigma}$-ring in $\Sp_{\cV}^{D_2}$. We define 
 \[ N_{D_2}^{O(2)}A= \cI_{\widetilde{\cV}}^{\cU}| B_{\bullet}^{\di} (A)  |.\]  This is functorial for $A\in \Assoc_\sigma(\Sp_{\cV}^{D_2})$.
 This defines a functor 
 \[ N_{D_2}^{O(2)}\colon \thinspace \Assoc_{\sigma}(\Sp_{\cV}^{D_2}) \to \Sp_{\cU}^{O(2)}.\]
\end{defin}
 \begin{rem}
 In particular, we produce a functor 
 \[ N_{D_2}^{O(2)}:\Comm(\Sp_{\cV}^{D_2})\longrightarrow \Comm(\Sp_{\cU}^{O(2)})\]
 by restriction to the subcategory $\Comm(\Sp_{\cV}^{D_2})\subset \Assoc_{\sigma}(\Sp_{\cV}^{D_2})$.
 \end{rem}

Note that the category of commutative monoids in $\Sp_{\cV}^{D_2}$ is tensored over the category of $D_2$-sets. This follows because $\Sp_{\cV}^{D_2}$ is tensored over $D_2$-sets and the forgetful functor 
$\Comm(\Sp_{\cV}^{D_2})\to \Sp_{\cV}^{D_2}$
 creates all indexed limits and the category $\Comm(\Sp_{\cV}^{D_2})$ contains all $\Top^{D_2}$-enriched equalizers, by a generalization of~\cite[Lem. 2.8]{MSV97}. We simply write $\otimes$ for this tensoring, viewed as a functor
 \[ -\otimes - \colon \thinspace  \Comm(\Sp_{\cV}^{D_2})\times D_2\text{-}\Set   \longrightarrow \Comm(\Sp_{\cV}^{D_2}).\]
 Note that this naturally extends to a functor 
 \[ -\otimes - \colon \thinspace  \Comm(\Sp_{\cV}^{D_2}) \times \left (D_2\text{-}\Set\right )^{\Delta^{\op}}\longrightarrow \Comm(\Sp_{\cV}^{D_2}).\]
\begin{defin}
Consider the minimal simplicial model
\[
	\xymatrix{ D_2 \ar[r] & 
 	\ar@<1ex>[l]\ar@<-1ex>[l]   \ar@<1ex>[r]\ar@<-1ex>[r]   D_4  &  \ar@<2ex>[l]\ar@<-2ex>[l]\ar[l]
                                 \ar@<2ex>[r]\ar@<-2ex>[r]\ar[r]  D_6 &
 	  \dots \ar@<3ex>[l]\ar@<-3ex>[l]\ar@<1ex>[l]\ar@<-1ex>[l]  \ldots  
  }\]
  for $O(2)$.  
Note that $D_{2(\bullet+1)}$ is    a dihedral set~\cite[Lemma~6.3.1]{Loday98}.  Therefore $\sq(D_{2(\bullet+1)}),$ its Segal-Quillen subdivision (as in~\cite{Seg73}), 
   is a simplicial $D_2$-set. We define 
  \[O(2)_{\bullet}=\sq(D_{2(\bullet+1)})\] as a simplicial $D_2$-set. 
  \end{defin} 

\begin{defin}\label{O2 tensor}
Let $A\in \Comm(\Sp_{\cV}^{D_2})$. Form the coequalizer of the diagram 
\[ 
\xymatrix{ 
 A \otimes D_2\otimes	O(2)_{\bullet}  \ar@<1ex>[rr]^(.5){\id_{A} \otimes \psi } \ar@<-1ex>[rr]_(.5){N\otimes\id_{O(2)_{\bullet}}} &&  A\otimes O(2)_{\bullet}
	}
\]
where 
$
	N \co A\otimes D_2 =N^{D_2}_e\iota_e^*A \to A 
$
is the $E_0$-$N^{D_2}_e\iota_e^*A$-algebra structure map and $\psi \colon \thinspace D_2\times O(2)_{\bullet}\longrightarrow O(2)_{\bullet}$ is the left $D_2$-action on $O(2)_{\bullet}$. 
We define the coequalizer in the category of simplicial objects in $\Comm(\Sp_{\cV}^{D_2})$ to be 
$  A  \otimes_{D_2}O(2)_{\bullet} $
and we may consider the geometric realization $|A\otimes_{D_2} O(2)_{\bullet}|$ as an object in $\Sp_{\widetilde{\cV}}^{O(2)}$. We therefore make the following definition
\begin{align*}
\label{relative tensor norm}
A\otimes_{D_2} O(2) =\cI_{\widetilde{\cV}}^{\cU}|A\otimes_{D_2} O(2)_{\bullet}|.
\end{align*}
\end{defin}

We now identify the norm (Definition~\ref{O2 norm}) with the tensor (Definition~\ref{O2 tensor}) in the commutative case. Recall that an orthogonal $G$-spectrum $X$ indexed on a complete universe $\cU$ is \emph{well-pointed} if $X(V)$ is well-pointed in $\Top^G$ for all finite dimensional orthogonal $G$-representations $V$.  We say an $E_{\sigma}$-ring in $\Sp^{D_2}_{\cV}$ is \emph{very well-pointed} if it is well-pointed and the unit map $S^0\to X(\mathbb{R}^0)$ is a Hurewicz cofibration in $\Top^{D_2}$.

\begin{prop}\label{equivalence of norm and tensor}
Let $A$ be an object in $\Comm(\Sp_{\cV}^{D_2})$ and assume that $A$ is very well-pointed. 
There is a natural map 
\[  N_{D_2}^{O(2)}A \longrightarrow   A  \otimes_{D_2} O(2),\]
in $\Comm(\Sp_{\cU}^{O(2)})$, which is a weak equivalence after forgetting to $\Sp_{\cV}^{D_2}$. 
\end{prop}
\begin{proof}
We first prove that there is an equivalence 
\[\sq (B_{\bullet}^{\di}(A)) \simeq A\otimes_{D_2} O(2)_{\bullet}. \]
 We consider the $k$ simplices on each side. On  the left, the $k$-simplices, are $A \wedge A^{\wedge \text{\bf 2k+1}}$
with $D_{4(k+1)}=\langle \omega ,t | t^{2(k+1)}=\omega^2=t\omega t\omega =1\rangle$ action given by letting $t$ cyclically permute the $2k+2$ copies of $A$, and letting $\tau$ act on $ \bfk=\{1,\dots ,2k+1\}$ by $\tau(i)=2k+1-i+1$. 
The $k$-simplices on the right hand side  are given by the coequalizer of the diagram 
\[
	\xymatrix{ 
       A  \otimes D_2 \otimes  D_{4(k+1)}   \ar@<1ex>[rr]^(.5){ \id_{A}  \otimes  \psi } \ar@<-1ex>[rr]_(.5){  N\otimes \id_{D_{4(k+1)}} } && A\otimes D_{4(k+1)}
		}
\]
in the category of $D_{4(k+1)}$-spectra. This is $ A \otimes_{D_2}D_{4(k+1)}$ and therefore the result on $k$-simplices follows from the $D_{4(k+1)}$-equivariant map 
\[  A \otimes_{D_2}  D_{4(k+1)}\simeq  A \wedge A^{\wedge \text{\bf 2k+1 \rm}}, \] 
which is clearly an equivalence on underlying commutative $D_2$-spectra.
To see that this map is $D_{4(k+1)}$-equivariant, note that $D_{4(k+1)}/D_2\cong \mu_{2k+2}$ as $D_{4(k+1)}$-sets and the right-hand-side can also be considered as a tensoring with the $D_{4(k+1)}$-set $\mu_{2k+2}$. Since this map is $D_{4(k+1)}$-equivariant it is compatible with the automorphisms in the dihedral category. It is also easy to check that this is compatible with the face and degeneracy maps. Since $A$ is very well-pointed, both sides are good in the sense of~\cite[Definition~1.5]{DMPPR21}, this level equivalence induces an equivalence on geometric realizations in the category of $D_2$-spectra by~\cite[Lemma~1.6]{DMPPR21}. 
\end{proof} 

 \begin{rem}
 Note that in the $\cF$-model structure on $\Sp_{\cW}^{G}$ where $G$ is a compact Lie group, $\cF$ is a family of  subgroups, and $\cW$ is a universe, the $\cF$-equivalences can either be taken to be the maps $X\to Y$ that induce isomorphisms on homotopy groups $\pi_*(X^H)\longrightarrow \pi_*(Y^H)$ for all $H\in \cF$ or the maps that induce isomorphisms  $\pi_*(\Phi^H X)\longrightarrow \pi_*(\Phi^HY)$, for all $H \in \cF$, where $\Phi^H$ denotes the $H$-geometric fixed points. 
 \end{rem}

 Recall that $\cR$ denotes the family of subgroups of $O(2)$ which intersect $\bT$ trivially.

\begin{cor}\label{geom fix points of norm}
Let $A$ be an object in $\Comm(\Sp_{\cV}^{D_2})$ and assume that $A$ is very well-pointed and flat in the sense of~\cite[Definition 2.6]{DMPPR21}).
Then there is an $\cR$-equivalence 
\[N_{D_2}^{O(2)}A\simeq  A \otimes_{D_2} O(2).\]
\end{cor}
\begin{proof}
Since there is a zig-zag of stable equivalences 
\[ \Phi^{D_2}(N_{D_2}^{O(2)}A)\simeq \Phi^{D_2}(\THR(A))\]
by~\cite[Remark 3.6]{DMP22} 
there is a zig-zag of stable equivalences 
\[ \Phi^{D_2}(N_{D_2}^{O(2)}A) \simeq \Phi^{D_2}(A) \wedge_{\iota_e^*A}^{\bL}\Phi^{D_2}(A)\]
where the right-hand-side is the derived smash product. 
Since $\Phi^{D_2}(-)$ sends homotopy colimits of  $O(2)$-orthogonal spectra indexed on $\cU$ to homotopy colimits of orthogonal spectra,\footnotemark \footnote{Using the Bousfield--Kan formula for homotopy colimits, we can write any homotopy colimit as the geometric realization of a simplicial spectrum and then use the fact that genuine geometric fixed points commute with sifted colimits.} we can identify 
$\Phi^{D_2}(A \otimes_{D_2} O(2)_{\bullet})$ with the homotopy coequalizer of
\[ 
	\xymatrix{ 
		\Phi^{D_2}(A \otimes D_2 \otimes O(2)_{\bullet} )\ar@<.5ex>[r] \ar@<-.5ex>[r]& \Phi^{D_2}(A\otimes  O(2)_{\bullet})
		}
\]
which is level-wise equivalent to 
\[ \
\Phi^{D_2}N^{D_{4(n+1)}}_{D_2}(A) \simeq \Phi^{D_2} \left( \bigwedge_{\gamma } N_{D_2\cap \gamma D_2\gamma^{-1}} ^{D_2}(\iota_{D_2\cap \gamma D_2\gamma^{-1}}^*c_{\gamma} A) \right )
\]
where $\gamma$ ranges over the representatives of double cosets $D_2\backslash \gamma \slash D_2$ in the set of double cosets $D_2\backslash D_{4(n+1)} / D_2$, which in turn is equivalent to 
\[ B_n(\Phi^{D_2}(A),\iota_e^*(A),\Phi^{D_2}(A)).\]
It is tedious, but routine, to check that this equivalence is compatible with the face and degeneracy maps. Since both source and target are Reedy cofibrant by our assumptions (cf.~\cite[Definition 1.5, Lemma 1.6]{DMPPR21}), 
this produces a stable equivalence 
\[ \Phi^{D_2}(A \otimes_{D_2}O(2))\simeq \Phi^{D_2}(A)\wedge_{\iota_e^*A}^{\bL}\Phi^{D_2}(A)\]
of orthogonal spectra on geometric realizations. 
Since all the groups of order $2$ in $\cR$ are conjugate in $O(2)$, we have proven the claim. 
\end{proof}  

\begin{lem}\label{lem: equivalence}
Given a stable equivalence of very well-pointed $E_{\sigma}$-rings $A\to A^{\prime}$ in $D_2$-orthogonal spectra indexed on $\cV$, there is an $\cR$-equivalence 
\[ N_{D_2}^{O(2)}A\to  N_{D_2}^{O(2)}A^{\prime}\]
of $O(2)$-spectra. 
\end{lem}
\begin{proof}
The induced map $ N_{D_2}^{O(2)}A \to N_{D_2}^{O(2)}A^{\prime}$
is $O(2)$-equivariant by construction, so it suffices to check that after restricting to $D_2$-spectra it is a stable equivalence of $D_2$-spectra. Note that this is independent of the choice of subgroup in $\cR$ of order $2$ because all of the subgroups of order two in $\cR$ are conjugate. 
After restricting to $D_2$-spectra, there is a zigzag of stable equivalences of $D_2$-spectra 
\[ \iota_{D_2}^*N_{D_2}^{O(2)}A \simeq  \iota_{D_2}^*\THR(A) \overset{\simeq}{\longrightarrow} \iota_{D_2}^*\THR(A^{\prime}) \simeq   \iota_{D_2}^*N_{D_2}^{O(2)}A^{\prime},\]
by~\cite[Theorem~2.20]{DMPPR21} and this agrees with the restriction of the map induced by $A\to A^{\prime}$ by naturality of~\cite[Remark 3.6]{DMP22}.
\end{proof}

\noindent We now show that  
the norm from $D_2$ to $O(2)$ (cf. Definition~\ref{O2 norm}) satisfies the universal property that one would expect of a norm. For a group $G$, let $All$ denote the family of all subgroups of $G$.
\begin{thm}\label{main norm thm}
The restriction  
\[ N_{D_2}^{O(2)} \co \Comm(\Sp_{\cV}^{D_2})\to \Comm(\Sp ^{O(2),\cR}_{\cU})\]
of the norm functor $N_{D_2}^{O(2)}$ to genuine commutative $D_2$-ring spectra is left Quillen adjoint to the restriction functor $\iota_{D_2}^*$ where 
\[\Comm(\Sp_{\cV}^{D_2}) \text{ and }\Comm(\Sp ^{O(2),\cR}_{\cU})\] 
are equipped with the $All$-model structure and the $\cR$-model structure  
respectively. 
\end{thm}
\begin{proof}
By Corollary~\ref{geom fix points of norm}, there is a natural $\cR$-equivalence
\[ N_{D_2}^{O(2)}(A)\simeq A\otimes_{D_2}O(2),  \]
of $O(2)$-orthogonal spectra. If $A\to A^{\prime}$ is an $\All$-equivalence of $D_2$-spectra where both source and target are very well-pointed then there is an $\cR$-equivalence
\[ N_{D_2}^{O(2)}(A)\simeq N_{D_2}^{O(2)}(A^{\prime})\]
of $O(2)$-spectra by Lemma~\ref{lem: equivalence}. In particular, if $A$ and $A^{\prime}$ are cofibrant in the positive complete stable model structure on $\Assoc_{\sigma}(\Sp_{\cV}^{D_2})$, then they are in particular very well-pointed.  This shows that both the functors $N_{D_2}^{O(2)}(-)$ and $(-)\otimes_{D_2}O(2)$ induce well-defined functors between the homotopy categories 
\[ N_{D_2}^{O(2)} \colon \thinspace \ho\left (\Comm(\Sp_{\cV}^{D_2})\right)\to \ho\left (\Comm(\Sp_{\cU}^{O(2),\cR}) \right )\]
and 
\[ - \otimes_{D_2} O(2) \colon \thinspace \ho \left (\Comm(\Sp_{\cV}^{D_2})\right)\to \ho\left (\Comm(\Sp_{\cU}^{O(2),\cR}) \right )\]
and they are naturally isomorphic on the homotopy categories. It is clear that $-\otimes_{D_2}O(2)$ is left adjoint to the restriction functor
\[ \iota_{D_2}^* \colon \thinspace \ho \left (\Comm(\Sp_{\cU}^{O(2),\cR}) \right )\to \ho\left (\Comm(\Sp_{\cV}^{D_2})\right).\]
Moreover, the restriction functor sends cofibrations and weak equivalences to cofibrations and weak equivalences by definition of the stable $\cR$-equivalences and the positive stable $\cR$-cofibrations. Consequently it also preserves all fibrations and acyclic fibrations. 
\end{proof}

\section{A multiplicative double coset formula}\label{mdcf}

The multiplicative double coset formula for finite groups gives an explicit formula for the restriction to $K$ of the norm from $H$ to $G$ where $H$ and $K$ are subgroups of $G$. For compact Lie groups, no such multiplicative double coset formula is known in general. In this section, we present a multiplicative double coset formula for the restriction to $D_{2m}$ of the norm from $D_2$ to $O(2)$. 

\begin{convention}\label{conv: ordered coset}
When the integer $m$ is understood from context, let
\[ \zeta=\zeta_{2m}=e^{2i\pi/2m}\in \bT\subset O(2).\]
We consider the element $\zeta$ as a lift of the element $-1$ along 
\[ D_{2m}\backslash O(2)/D_2\cong \mu_m\backslash \bT\cong \bT.\]
We make this choice of homeomorphism simply so that the formula for $\zeta$ can be chosen consistently for all $m$ independent of whether $m$ is odd or even. We observe that $\zeta D_{2}\zeta^{-1}=\langle \zeta_{m}\tau\rangle$.  

Fix total orders on the $D_{2m}$-sets $D_{2m}/e$, $D_{2m}/D_2$, and $D_{2m}/\zeta D_2 \zeta^{-1}$. Let
\[D_{2m}/e=\{1\le  \zeta_m\tau\le \zeta_m \le \zeta_m^{2}\tau \le  \zeta_m^2\le \dots \le \zeta_m^{m-1}\le \tau\},\]
\[ D_{2m}/D_2 = \{D_2\le \zeta_m D_2\le \dots \le  \zeta_m^{m-1}D_2\},\]
\[ D_{2m}/\zeta D_2\zeta^{-1} = \{\zeta D_2\zeta^{-1}\le \zeta_m \cdot \zeta  D_2\zeta^{-1}\le \dots \le \zeta_m^{m-1}\cdot\zeta D_2\zeta^{-1}\}.\]
These choices of total orderings on the $D_{2m}$-sets $D_{2m}/e$, $D_{2m}/D_2$, and  $D_{2m}/\zeta D_2\zeta^{-1}$ also fix group homomorphisms 
\[\lambda_{e} \colon \thinspace D_{2m}\to \Sigma_{2m},\] 
\[\lambda_{D_2} \colon \thinspace D_{2m} \to \Sigma_m \wr  D_2 \] 
\[\lambda_{\zeta D_2\zeta^{-1}}\colon \thinspace D_{2m}  \to \Sigma_m\wr \zeta D_2\zeta^{-1}.\] 
Denote the associated norms by $N_e^{D_{2m}}$, $N_{D_2}^{D_{2m}}$ and $N_{\zeta D_2\zeta^{-1}}^{D_{2m}}$ respectively.\footnote{The choice of ordering does not matter for our norm functors up to canonical natural isomorphism, but remembering the choice of ordering clarifies our constructions later.}
\end{convention}
\begin{remark}\label{order on mum}
Any finite subset $F$ of $\bT\subset \mathbb{C}$ can be equipped with a total order by considering $1\le e^{2i\pi\theta}\le e^{2i\pi\theta^{\prime}}$ for all $0\le\theta \le \theta^{\prime}<1$. In particular, the subset of $m$-th roots of unity $\mu_m$ in $\bT$ can be equipped with a total order. 
\end{remark}

\begin{lem}\label{lem: isomorphism of D2m sets}
There is an isomorphism of totally ordered $D_{2m}$-sets
\[f_{m,k} \colon \thinspace \mu_{2m(k+1)}\cong D_{2m}/D_2\amalg D_{2m}^{\amalg k} \amalg D_{2m}/\zeta D_2\zeta^{-1}\]
where $\zeta=\zeta_{2m}$, the $D_{2m}$-sets on the right have the total orders from Convention~\ref{conv: ordered coset}, and $\mu_{2m(k+1)}$ is equipped with a total order by Remark~\ref{order on mum}.
\end{lem}
\begin{proof}
Without	the total ordering this isomorphism is clear. The total ordering in Convention~\ref{conv: ordered coset} was chosen so that this lemma would be true. 
\end{proof}

\begin{rem}\label{rem: map of totally ordered sets}
We will also write $f_{m,k}$ for the underlying map of totally ordered sets from Lemma~\ref{lem: isomorphism of D2m sets} after forgetting the $D_{2m}$-set structure. 
\end{rem}
\noindent Given an $E_{\sigma}$-ring $R$, then $\iota_e^*R$ is an $E_1$-ring and $R$ is a $N_{e}^{D_2}\iota_e^*R$-bimodule with right action 
\[ \overline{\psi}_R\colon \thinspace R\wedge N_e^{D_2}\iota_e^*R\longrightarrow R\]
and left action 
\[ \overline{\psi}_L\colon \thinspace N_e^{D_2}\iota_e^*R\wedge R\longrightarrow  R.\]
We also note that there is an equivalence of categories 
\[ c_{\zeta}\colon \thinspace \Sp^{D_2} \longrightarrow  \Sp^{\zeta D_{2}\zeta^{-1}}\]
which is symmetric monoidal and therefore sends $E_{\sigma}$-rings in $\Sp^{D_2}$  to $E_{\sigma}$-rings in $\Sp^{\zeta D_{2}\zeta^{-1}}$. In particular, $c_{\zeta}R$ is a left $N_{e}^{\zeta D_2\zeta^{-1}}\iota_e^*R$-module with 
\[ c_{\zeta}(\overline{\psi}_L)\colon \thinspace N_e^{\zeta D_2\zeta^{-1}}\iota_e^*R\wedge c_{\zeta} R\longrightarrow c_{\zeta }R.\]

\begin{defin}\label{def:twistedmodule}
Let $R$ be an $E_{\sigma}$-ring. We define a right $N_{e}^{D_{2m}}\iota_e^*R$-module structure on $N_{D_2}^{D_{2m}}R$ as the composite
\[
	\xymatrix{
		 \psi_R\co  N_{D_2}^{D_{2m}}R \wedge N_{e}^{D_{2m}}\iota_e^*R \ar[r]^-{\cong} & N_{D_2}^{D_{2m}}( R\wedge N_{e}^{D_2}\iota_e^*R )   \ar[rr]^(.6){N_{D_2}^{D_{2m}}(\bar{\psi}_R)} && N_{D_2}^{D_{2m}}R.
		 } 
\]
We define a left $N_{e}^{D_{2m}}\iota_e^*R$-module structure on $N_{\zeta D_2 \zeta^{-1}}^{D_{2m}}c_{\zeta}R$
\[\psi_L\colon \thinspace N_e^{D_{2m}}\iota_e^*R\wedge N_{\zeta D_2 \zeta^{-1}}^{D_{2m}}c_{\zeta}R\longrightarrow N_{\zeta D_2 \zeta^{-1}}^{D_{2m}}c_{\zeta}R\] 
as the composite of the isomorphism
\begin{align*}
\label{twisted left module}	
	\xymatrix{
		 N_{e}^{D_{2m}}\iota_e^*R\wedge N_{\zeta D_2\zeta^{-1}}^{D_{2m}}c_{\zeta}R\ar[r]^(.45){\cong} &
		 N_{\zeta D_2\zeta^{-1}}^{D_{2m}}\left (N_e^{\zeta D_2\zeta^{-1}}\iota_e^*R\wedge c_{\zeta}R\right )}
\end{align*}
with the map
\[ 
	\xymatrix{
		N_{\zeta D_2\zeta^{-1}}^{D_{2m}}(N_e^{\zeta D_2\zeta^{-1}}\iota_e^*R \wedge c_{\zeta}R) \ar[rrr]^(.6){N_{\zeta D_2\zeta^{-1}}^{D_{2m}}(c_{\zeta}(\bar{\psi}_L))}&&& N_{\zeta D_2\zeta^{-1}}^{D_{2m}}c_{\zeta}R.
		}
\]
\end{defin}
\begin{exm}
When $m=2$, we note that $\zeta_{4}D_2\zeta_{4}^{-1}$ in $D_{8}$ can be identified with the diagonal subgroup $\triangle$ of $D_{4}$. In this case, $\triangle$ and $D_2$ are conjugate in $D_{8}$ even though they are not conjugate in $D_4$. We still define a left $N_e^{D_{4}}\iota_e^*R$-module structure on $N_{\triangle}^{D_{4}}c_{\zeta}R$ by composing the map
\begin{align*}
\label{twisted left module 2}	
	\xymatrix{
		 N_{e}^{D_{4}}\iota_e^*R\wedge N_{\triangle}^{D_{4}}c_{\zeta}R\ar[r]^-{\cong} &   N_{\triangle}^{D_{4}}\left (N_e^{\triangle}\iota_e^*R\wedge c_{\zeta}R   \right )
		 }
\end{align*}
with the map 
\[ 
	\xymatrix{
		N_{\triangle}^{D_{4}}(N_e^{\triangle}\iota_e^*R \wedge c_{\zeta}R) \ar[rrr]^(.6){N_{\triangle}^{D_{4}}(c_{\zeta}(\bar{\psi}_L))}&&& N_{\triangle}^{D_{4}}c_{\zeta}R.
		}
\]
\end{exm}

\begin{rem}\label{rem:simplicialD2msets}
Note that there is an isomorphism of simplicial $D_{2m}$-sets 
\[ \mu_{2m(\bullet+1)}\to D_{2m}/D_2 \amalg D_{2m}^{\amalg \bullet} \amalg D_{2m}/\zeta D_{2}\zeta^{-1}\]
which is given by the isomorphism $f_{m,k}$ of totally ordered $D_{2m}$-sets of Lemma~\ref{lem: isomorphism of D2m sets} on $k$-simplices. 
The simplicial maps on the left are given by the simplicial maps in the simplicial set $\text{sd}_{D_{2m}}S^1_{\bullet}$ where $S^1_{\bullet}$ is the minimal model of $S^1$ as a simplicial set. On the right, the face maps 
\[D_{2m}/D_2 \amalg D_{2m}^{\amalg k} \amalg  D_{2m}/\zeta D_2\zeta^{-1}\to  D_{2m}/D_2 \amalg D_{2m}^{\amalg k-1} \amalg D_{2m}/\zeta D_2\zeta^{-1}\] 
are given by the canonical quotient composed with the fold map 
\[D_{2m}/D_2\amalg D_{2m} \overset{D_{2m}/D_2\amalg q}{\longrightarrow} D_{2m}/D_2\amalg D_{2m}/D_2 \overset{\nabla}{\longrightarrow} D_{2m}/D_2\] 
for the first face map, the fold map
\[ D_{2m} \amalg D_{2m} \overset{\nabla}{\longrightarrow} D_{2m} \]
for the middle maps, and for the last face map it is given by the composite 
\[ D_{2m}\amalg D_{2m}/\zeta D_2\zeta^{-1}  \overset{q\amalg 1}{\longrightarrow} D_{2m}/\zeta D_2\zeta^{-1} \amalg D_{2m}/\zeta D_2\zeta^{-1} \overset{\nabla}{\longrightarrow}  D_{2m}/\zeta D_2\zeta^{-1}.\]
The degeneracy maps are given by the canonical inclusions. 
\end{rem}
 
\begin{prop}\label{charofdibar}
Suppose $R$ is an $E_{\sigma}$-ring in $D_{2}$-spectra 
indexed on the complete universe $\cV=\iota_{D_{2}}^*\cU$ where $\cU$ is a fixed complete $O(2)$-universe. More generally, let $\cV_n=\iota_{D_{2n}}^*\cU$ for $n\ge 1$. 
There is an isomorphism of simplicial $D_{2m}$-spectra 
\[
	\cI_{\widetilde{\cV}}^{\cV_m}(\operatorname{sd}_{D_{2m}}B^{\di}_{\bullet}(R)) \cong B_{\bullet} \left (N_{D_2}^{D_{2m}}R,N_{e}^{D_{2m}}\iota_e^*R,N_{\zeta D_2\zeta^{-1}}^{D_{2m}}c_{\zeta}R\right ).
\]
where we write $\widetilde{\cV}$ for the $D_{2}$-universe $\cV$ regarded as a $D_{2m}$-universe via inflation along the canonical quotient $D_{2m}\longrightarrow D_2$. 
\end{prop}
\begin{proof}
When $R$ is an $E_{\sigma}$-ring, we explicitly define the simplicial map 
\[\cI_{\widetilde{\cV}}^{\cV_m}\operatorname{sd}_{D_{2m}}B^{\di}_{\bullet}(R) \to B_{\bullet }(N_{D_2}^{D_{2m}}R,N_{e}^{D_{2m}}\iota_e^*R,N_{\zeta D_2 \zeta^{-1}}^{D_{2m}}c_{\zeta}R) )\]
on $k$-simplices. There is an isomorphism given by composition of two maps. The first map 
\[  
 \xymatrix{
 f_{m,k} \colon \thinspace R^{\wedge \mu_{2m(k+1)}} \ar[r] & R^{\wedge m} \wedge (R^{\wedge m } \wedge (R^{\op})^{\wedge m})^{\wedge k} \wedge (R^{\op})^{\wedge m}
 }
\]
is the isomorphism induced by $f_{m,k}$ of Remark~\ref{rem: map of totally ordered sets} regarded simply as a map of totally ordered sets. The second map is  
\[ 
\xymatrix{ 
	R^{\wedge m} \wedge \left ((R\wedge R^{\op})^{\wedge m}\right )^{\wedge k} \wedge (R^{\op})^{\wedge m}  \ar[d]^(.5){1^{\wedge m}  \wedge  (1\wedge \tau )^{\wedge mk}\wedge \tau^{\wedge m} } \\ 
	N_{D_2}^{D_{2m}}R \wedge  (N_{e}^{D_{2m}}R)^{\wedge k} \wedge N_{\zeta D_2 \zeta^{-1}}^{D_{2m}}c_{\zeta}R,
	}
 \]
where we use Convention~\ref{conv: ordered coset}. This is a $D_{2m}$-equivariant isomorphism on $k$-simplices essentially because it comes from the isomorphism of ordered $D_{2m}$-sets of Lemma~\ref{lem: isomorphism of D2m sets}.
It follows that the map is compatible with the simplicial structure maps by comparing the structure maps in the isomorphism of simplicial $D_{2m}$-sets in Remark~\ref{rem:simplicialD2msets} to the structure maps on either side. 
\end{proof}
\noindent Consequently, for a flat $E_{\sigma}$-ring in the sense of~\cite[Definition 2.6]{DMPPR21}), we have the following multiplicative double coset formula. This generalizes the $m=1$ case appearing in~\cite{DMPPR21}.
\begin{thm}[Multiplicative Double Coset Formula]\label{thm:doublecoset}
When $R$ is a flat $E_{\sigma}$-ring and $m>0$, there is a stable equivalence of $D_{2m}$-spectra 
\[ \iota_{D_{2m}}^*N_{D_2}^{O(2)}R \simeq  N_{D_2}^{D_{2m}}R \wedge^{\bL}_{N_e^{D_{2m}}\iota_e^*R} N_{\zeta D_2\zeta^{-1}}^{D_{2m}}c_{\zeta}R. \]
\end{thm}
 
\begin{proof}
By Proposition~\ref{charofdibar}, we know there is an equivalence
\[
	\cI_{\widetilde{\cV}}^{\cV_n}(\operatorname{sd}_{D_{2m}}B^{\di}_{\bullet}(R))\cong B_{\bullet} (N_{D_2}^{D_{2m}}R,N_{e}^{D_{2m}}\iota_e^*R, N_{\zeta D_2\zeta^{-1}}^{D_{2m}}c_{\zeta}R).
\]
When $R$ is a flat $E_{\sigma}$-algebra in $D_2$-spectra in the sense of~\cite[Definition 2.6]{DMPPR21}), then $N_{D_2}^{D_{2m}}R$ is a flat  $N_e^{D_{2m}}\iota_e^*R$-module by~\cite[Theorem 3.4.22-23]{Sto11} and~\cite{BDS18} and therefore we may identify
 \[N_{D_2}^{D_{2m}}R \wedge^{\bL}_{N_e^{D_{2m}}\iota_e^*R}  N_{\zeta D_2\zeta^{-1}}^{D_{2m}}c_{\zeta}R\] 
 with the realization of the bar resolution of 
 $N_{D_2}^{D_{2m}}R$ by free $N_e^{D_{2m}}\iota_e^*R$-modules then smashed with the right $N_e^{D_{2m}}\iota_e^*R$-module $N_{\zeta D_2\zeta^{-1}}^{D_{2m}}c_{\zeta}R$. This is exactly
the realization of the simplicial spectrum 
\[B_{\bullet}(N_{D_2}^{D_{2m}}R,N_{e}^{D_{2m}}\iota_e^*R, N_{\zeta D_2\zeta^{-1}}^{D_{2m}}R).\] 
\end{proof}

\section{Real Hochschild homology}\label{sec: HR}

 In this section we address the question: What is the algebraic analogue of THR? We do this by defining a theory of Real Hochschild homology for discrete $E_{\sigma}$-rings. We then show how this leads to a theory of Witt vectors for rings with anti-involution. 
 
To begin, we recall some basic terminology in the theory of Mackey functors, we define norms in the category of Mackey functors, and $E_{\sigma}$-algebras in $D_2$-Mackey functors, which we call discrete $E_{\sigma}$-rings.  
\subsection{Mackey functors and norms}\label{rep functors}
See~\cite[\S 2]{BGHL19} for a more thorough review of the theory of Mackey functors. Here we simply recall the contructions and notation we use in the present paper. Let $G$ be a finite group. Write $\cA$ for the Burnside category of $G$.  For a finite $G$-set $X$,
$\mA^G_{X}:=\cA(X,-)$
denotes the representable $G$-Mackey functor represented by $X$. 
This construction forms a co-Mackey functor object in Mackey functors, by viewing it also as a functor in the variable $X$, so in particular
\begin{align*}\label{representable Mackey functor property}
	\mA^G_{X\amalg Y}=\cA(X\amalg Y,-)=\cA(X,-)\oplus\cA(Y,-)=\mA^G_X\oplus \mA^G_Y.
\end{align*}
Write $\mA^G$ for the Burnside Mackey functor associated to $G$, which can be identified with $\mA^G_{*}$ where $*=G/G$. This Mackey functor has the property that $\mA^G(G/H)=A(H)$
where $A(H)$ denotes the Burnside ring for a finite group $H$. Recall that as an abelian group $A(H)$ is free with basis $\{ [H/K] \}$ where $K$ ranges over all conjugacy classes of subgroups $K\le H$. When $K=H$ we simply write $1=[H/H]$ and when $K$ is the trivial group we simply write $[H]=[H/\{e\}]$. The transfer and restriction maps in $\mA^G$ are given by induction and restriction maps on finite sets.  
\begin{exm}\label{ex: burnside c2}
The Burnside Mackey functor $\mA^{D_2}$ can be described by the following diagram. 
\[\xymatrix{ 1 \ar@{|->}[d] & [D_2] \ar@{|->}[d] & \m{M}^{D_2}(D_2/D_2)=\mathbb{Z}\langle1, [D_2] \rangle  \ar@/_1pc/[d]_{res_e^{D_2}} & [D_2]  \\
1 & 2 & \m{M}^{D_2}(D_2/*)=\mathbb{Z}\langle 1 \rangle \ar@/_1pc/[u]_{tr_e^{
D_2}} & 1 \ar@{|->}[u]
}\]	
\end{exm}
\noindent Given a finite group $G$, a subgroup $H\le G$, and a $H$-set $X$ we write $\Map^H(G,X)$
for the $G$-set of $H$-equivariant maps from $G$ to $X$, which is a functor in the variable $X$ known as coinduction.

Recall that these categories $\Sp_{\cU}^G$ of $G$-spectra indexed on a complete universe $\cU$ and $\Mak_G$ of $G$-Mackey functors are both symmetric monoidal, and the symmetric monoidal structures are compatible in the following sense.
 \begin{prop}\cite{LM06}\label{prop:smabox}
 For $X$ and $Y$ cofibrant, (-1)-connected orthogonal $G$-spectra, there is a natural isomorphism
 \[
 \m{\pi}_0(X \wedge Y) \cong \m{\pi}_0 X \square \m{\pi}_0 Y.
 \]
 \end{prop}
 A $G$-Mackey functor $\m{M}$ has an associated Eilenberg--MacLane $G$-spectrum, $H\m{M}$. The defining property of this spectrum is that
 \[
 \m{\pi}^G_k(H\m{M}) \cong \left\{ \begin{array}{ll} \m{M} &\text{ if } k = 0 \\
 0 &\text{ if } k \neq 0. \\
 \end{array} \right.
 \]
 It then follows from Proposition~\ref{prop:smabox} above that the box product of Mackey functors has a homotopical description:
 \[
 \m{M} \square \m{N} \cong \m{\pi}_0(H\m{M} \wedge H\m{N}). 
 \]

 The category $\Sp_{\cU}^G$ has an equivariant enrichment of the symmetric monoidal product, a $G$-symmetric monoidal category structure~\cite{HHR,HillHopkins16}. Such a $G$-symmetric monoidal structure requires multiplicative norms for all subgroups $H\le G$. In $\Sp_{\cU}^G$ these are given by the Hill--Hopkins--Ravenel norm. The $G$-symmetric monoidal structure on $\Sp_{\cU}^G$ induces such a structure on Mack$_G$ as well. In particular, one can define norms for $G$-Mackey functors. 

 \begin{defin}[cf. {\cite{HillHopkins16}}]\label{def: norm in Mackey}
 Given a finite group $G$ with subgroup $H$ and an $H$-Mackey functor $\mM$, the norm in Mackey functors is defined by 
 \[ N_{H}^{G}\mM = \upi_0^{G}N_H^G H\mM.\]
 \end{defin}
\noindent These Mackey functor norms also appear under a different guise in earlier work of Bouc~\cite{Bouc}.
 
 The following lemma is immediate. 
 \begin{lem}\label{norm of sifted colimit}
 The norm in Mackey functors commutes with sifted colimits. 
 \end{lem}
 
\subsection{Discrete \texorpdfstring{$E_{\sigma}$}{Esigma}-rings}
 In Section~\ref{sec: norm}, we discussed $E_{\sigma}$-rings in $D_2$-spectra, which serve as the input for Real topological Hochschild homology. We now define their algebraic analogues, discrete $E_{\sigma}$-rings.  These discrete $E_{\sigma}$-rings will be the input for our construction of Real Hochschild homology. 
 
\begin{defin}\label{EValgebras}
Let $V$ be a finite dimensional representation of a finite group $G$. An $E_{V}$-algebra in $G$-Mackey functors is a $\mathcal{P}_V$-algebra in $G$-Mackey functors, where $\mathcal{P}_V$ is the monad
 \[\mathcal{P}_V(-)=\bigoplus_{n\ge 0} \underline{\pi}_0^{G}\left ( ((E_{V,n})_{+}\wedge_{\Sigma_n} H(-)^{\wedge n}\right ) \]
 and $H(-)$ is the Eilenberg--MacLane funtor.
\end{defin}
\noindent When \(V=\sigma\), this monad is particularly simple, since the spaces in the \(E_{\sigma}\)-operad are homotopy discrete. 

\begin{prop}
For \(D_{2}\)-Mackey functors, the monad \(\mathcal P_{\sigma}\) is given by
\[
\mathcal P_{\sigma}(\mM)=T\big(N_{e}^{D_{2}}i_{e}^{\ast}\mM)\square (\m{A}\oplus \mM),
\]
where \(T(-)\) is the free associative algebra functor.
\end{prop}
\begin{proof}
Recall that we have an equivariant equivalence
\[
E_{\sigma,n}\simeq (D_{2}\times\Sigma_{n})/\Gamma_{n}. 
\]
If \(n\) is even, then we have natural isomorphisms
\[
\m{\pi}_{0}\big(((E_{\sigma,n})_+\wedge_{\Sigma_{n}} H\mM^{\wedge n}\big)\cong \m{\pi}_{0} \big((N^{D_{2}}_e \iota_e^* H\mM)^{\wedge n/2}\big)\cong (N^{D_{2}}_e\iota_e^*\mM)^{\square n/2}.
\]
If \(n\) is odd, then since fixed points contributes a box-factor of \(\mM\) itself:
\[
\m{\pi}_{0}\big(((E_{\sigma,n})_+\wedge_{\Sigma_{n}} H\mM^{\wedge n}\big)\cong  (N^{D_{2}}_e\iota_e^*\mM)^{\square \lfloor n/2\rfloor}\square \mM.
\]
The result follows from grouping the terms according to the number of box-factors involving the norm.
\end{proof}

\begin{defin}\label{defn:discreteEsigma}
By a \emph{discrete \(E_{\sigma}\)-ring}, we mean an algebra over the monad \(\mathcal P_{\sigma}\) in the category of $D_2$-Mackey functors. 
\end{defin}

We can further unpack this structure to describe the monoids.

\begin{lem}\label{Esigmastructure}
A discrete \(E_{\sigma}\)-ring is the following data: 
\begin{enumerate}
\item A $D_2$-Mackey functor \(\mM\), together with an associative product on \(\mM(D_{2}/e)\) for which the Weyl action is an anti-homomorphism,
\item a \(N_{e}^{D_{2}}\iota_e^*\mM\)-bimodule structure on \(\mM\) that restricts to the standard action of \(\mM(D_{2}/e)\otimes \mM(D_{2}/e)^{\textup{op}}\) on \(\mM(D_{2}/e)\).
\item an element \(1\in \mM(D_{2}/D_{2})\) that restricts to the element \(1\in\mM(D_{2}/e)\).
\end{enumerate}
\end{lem}
\begin{rem}
These conditions are almost 
those of a Hermitian Mackey functor in the sense of~\cite{DO19}: the only difference is that a discrete $E_{\sigma}$-ring includes the additional assumption that there is a fixed unit element $1\in \mM(D_2/D_2)$, or in other words, an $E_0$-$N^{D_2}_e\iota_e^*\mM(D_2/e)$-ring structure on $\mM$.  
\end{rem}
\begin{exm}\label{exm: pi0 of Esigma alg}
Given an $E_{\sigma}$-ring $R$ it is clear that $\m{\pi}_0^{D_2}(R)$ is a discrete $E_{\sigma}$-ring. In fact, this does not depend on our choice of $E_{\sigma}$-operad. By Example~\ref{exm: commutative implies sigma}, we also conclude that if $R$ is a commutative monoid in $\Sp_{\cV}^{D_2}$, then $\upi_0^{D_2}R$ is a discrete $E_\sigma$-ring. 
\end{exm}

\begin{exm}[Rings with anti-involution]
Let $R$ be a discrete ring with anti-involution $\tau \co R^{\op}\to R$, regarded as the action of the generator of $D_2$. Then there is an associated Mackey functor $\mM$ with $\mM(D_2/e)=R$ and $\mM(D_2/D_2)=R^{C_2}$. The restriction map $\res_e^{D_2}$ is the inclusion of fixed points, the transfer $\tr_e^{D_2}$ is the map $1+\tau$, and the Weyl group action of $D_2$ on $\mM(D_2/e)=R$ is defined on the generator of $D_2$  by the anti-involution $\tau\co R^{\op}\longrightarrow R$. Since $R^{\op}\to R$ is a ring map, there is an element $1\in R^{C_2}$ that restricts to the multiplicative unit in $ 1\in R$. This specifies a discrete $E_{\sigma}$-ring structure on the Mackey functor $\mM$. 
\end{exm}

\subsection{Real Hochschild homology of discrete \texorpdfstring{$E_{\sigma}$}{Esigma}-rings}\label{sec:HR}
In this section, we define the Real Hochschild homology $\HR_*^{D_{2m}}(\m{M})$ of a discrete $E_{\sigma}$-ring $\m{M}$, which takes values in graded $D_{2m}$-Mackey functors. 
We first need to specify a right $N_e^{D_{2m}}\iota_e^*\m{M}$-action on  $N_{D_2}^{D_{2m}}\m{M}$, and a left $N_e^{D_{2m}}\iota_e^*\m{M}$-action on  $N_{\zeta D_2 \zeta^{-1}}^{D_{2m}}c_{\zeta}\m{M}$. Here $c_{\zeta}$ is the symmetric monoidal equivalence of categories 
\[ c_{\zeta} \colon \thinspace \Mak_{D_2}\to \Mak_{\zeta D_2 \zeta^{-1}}.\]
For $\m{M}$ a discrete $E_{\sigma}$-ring, $\iota_e^*\m{M}$ is an (associative unital) ring  so $N_e^{D_{2m}}\iota_e^*\m{M}$ is an associative Green functor. By Lemma~\ref{Esigmastructure}, there is a left action
\[ \overline{\psi}_L \co N_e^{D_2}\iota^*_e\m{M} \square \m{M} \to  \m{M} \]
and a right action
\[ \overline{\psi}_R \co \m{M}  \square N_e^{D_2}\iota^*_e\m{M} \to  \m{M}.\]

\begin{defin}\label{def:twistedMackeymodule}
We define a right $N_e^{D_{2m}}\iota_e^*\m{M}$-module structure on $N_{D_2}^{D_{2m}}\m{M}$ as the composite
\[
	\xymatrix{
		 \psi_R\co  N_{D_2}^{D_{2m}}\m{M} \square N_{e}^{D_{2m}}\iota_e^*\m{M} \ar[r]^-{\cong} & N_{D_2}^{D_{2m}}( \m{M}\square N_{e}^{D_2}\iota_e^*\m{M} )   \ar[rr]^(.6){N_{D_2}^{D_{2m}}(\overline{\psi}_R)} && N_{D_2}^{D_{2m}}\m{M}.
		 } 
\]
  We define a left $N_e^{D_{2m}}\iota_e^*\m{M}$-module structure $\psi_L$ on  $N_{\zeta D_2 \zeta^{-1}}^{D_{2m}}c_{\zeta}\m{M}$
as the composite of the map
\begin{align*}
\label{twisted left module}	
	\xymatrix{
		 N_{e}^{D_{2m}}\iota_e^*\m{M}\square N_{\zeta D_2 \zeta^{-1}}^{D_{2m}}c_{\zeta}\m{M} \ar[r]^-{\cong} & N_{\zeta D_2 \zeta^{-1}}^{D_{2m}} ( N_e^{\zeta D_2\zeta^{-1}}\iota_e^*\m{M}\square c_{\zeta}\m{M})
		 }
\end{align*}
with the map 
\[ 
	\xymatrix{
		N_{\zeta D_2\zeta^{-1}}^{D_{2m}}(N_e^{\zeta D_2\zeta^{-1}}\iota_e^*\m{M} \square c_{\zeta}\m{M}) \ar[rrr]^(.6){N_{\zeta D_2\zeta^{-1}}^{D_{2m}}(c_{\zeta}(\overline{\psi}_L))}&&& N_{\zeta D_2 \zeta^{-1}}^{D_{2m}}c_{\zeta}\m{M}.
		}
\]
where $c_{\zeta}(\overline{\psi}_L)$ is the left action of $N_e^{\zeta D_2 \zeta^{-1}}\iota_e^*\m{M}$ on $c_{\zeta}\m{M}$ coming from the fact that $c_{\zeta}$ is symmetric monoidal and therefore sends $E_{\sigma}$-rings in $\Mak_{D_2}$ to $E_{\sigma}$-rings in $\Mak_{\zeta D_2 \zeta^{-1}}$.
\end{defin}
\begin{defin}
Given Mackey functors $\m{R}$, $\mM$, and $\mN$ where $\m{R}$ is an associative Green functor, $\mN$ is a right $\m{R}$-module and $\mN$ is a left $\m{R}$-module, we define the two-sided bar construction 
$B_{\bullet}(\mM,\m{R},\mN)$
with $k$-simplices 
\[ B_{k}(\mM,\m{R},\mN) = \mM \square \m{R}^{\square k} \square \mN \]
and the usual face and degeneracy maps. 
\end{defin}

\begin{defin}\label{def: HR}
The \emph{Real $D_{2m}$-Hochschild homology} of a discrete $E_{\sigma}$-ring $\underline{M}$ is defined to be the graded $D_{2m}$-Mackey functor
\[ \HR_*^{D_{2m}}(\m{M}) = H_*\left ( \HR_{\bullet}^{D_{2m}}(\m{M}) \right ). \]
where 
\[ \HR_{\bullet}^{D_{2m}}(\m{M}) =B_{\bullet}(N_{D_2}^{D_{2m}}\m{M},N_{e}^{D_{2m}}\iota_e^*\m{M}, N_{\zeta D_2 \zeta^{-1}}^{D_{2m}}c_{\zeta}\m{M} ).\]
\end{defin}
\begin{remark}
Given a right $N_e^{D_{2m}}i_e^*\underline{M}$-module $\underline{N}$, we can also define Real Hochschild homology with coefficients
\[ \underline{\mathrm{HR}}_0^{D_{2m}}(\underline{M};\underline{N})=H_*(B_{\bullet}(\underline{N},N_e^{D_{2m}}i_e^*\underline{M}, N_{\zeta D_2 \zeta^{-1}}^{D_{2m}}c_{\zeta}\m{M} ).
\]
\end{remark}
Recall that the homology of a simplicial Mackey functor is defined to be the homology of the associated normalized dg Mackey functor, as in~\cite{BGHL19}.
\begin{lem}\label{lem: HR0 of Tambara is Tambara}
If $\mM$ is a $D_2$-Tambara functor, then $\HR_{0}^{D_{2m}}(\m{M})$ is a $D_{2m}$-Tambara functor.
\end{lem}
\begin{proof}
Reflexive coequalizers in the category of  Tambara functors are computed as the reflexive coequalizer of the underlying Mackey functors~\cite{Strickland12}. 
\end{proof} 

\begin{prop}\label{prop:HRBox}
There is an isomorphism of $D_{2m}$-Mackey functors 
\[ \HR_0^{D_{2m}}(\m{M}) \cong N_{D_2}^{D_{2m}}\m{M}\square_{N_e^{D_{2m}}\iota_e^*\m{M}}N_{\zeta D_2 \zeta^{-1}}^{D_{2m}}c_{\zeta}\m{M}.  \]
\end{prop}
\begin{proof}
Both sides are given by the coequalizer 
\begin{align*} 
	\xymatrix{ 
		N_{D_2}^{D_{2m}}\m{M} \square N_{\zeta D_2 \zeta^{-1}}^{D_{2m}}c_{\zeta}\mM &  \ar@<.5ex>[l]^-{\psi_R \square \id}\ar@<-.5ex>[l]_-{\id \square \psi_L} N_{D_2}^{D_{2m}}\m{M}\square N_e^{D_{2m}}\iota_e^*\m{M} \square N_{\zeta D_2 \zeta^{-1}}^{D_{2m}}c_{\zeta}\m{M}.
		}
\end{align*}
\end{proof}

\begin{defin}
We say that a Mackey functor $\mM$ is \emph{flat} if the derived functors of the functor $\mM\square-$ vanish. 
\end{defin}
\begin{prop}\label{prop HR}
For any discrete $E_{\sigma}$-ring $\m{M}$ that can be written as a filtered colimit of representable $D_2$-Mackey functors, there is an isomorphism 
\[ \HR_*^{D_{2m}}(\mM) =\m{\Tor}_*^{N_e^{D_{2m}}\iota_e^*\mM} (N_{D_2}^{D_{2m}}\mM  , N_{\zeta D_2 \zeta^{-1}}^{D_{2m}}c_{\zeta}\mM ). \]
\end{prop}
\begin{proof}
Norms send representable $D_2$-Mackey functors to representable $D_{2m}$-Mackey functors by~\cite[Proposition 3.7]{BGHL19}. Norms also commute with sifted colimits. Filtered colimits of representable $D_{2m}$-Mackey functors are flat.
\end{proof}

\subsection{Comparison between Real Hochschild homology and THR}

We now show Real topological Hochschild homology and Real Hochschild homology are related by a linearization map which is an isomorphism in degree 0. 
\begin{thm}\label{thm:linearization}
For any $(-1)$-connected  $E_\sigma$-ring $A$, we have a natural homomorphism
\[
\m{\pi}_k^{D_{2m}} \THR(A) \longrightarrow  \HR_k^{D_{2m}}(\m{\pi}_0^{D_2}A). 
\]
which is an isomorphism when $k=0$. 
\end{thm}
\begin{proof}
By~\cite[X.2.9]{EKMM97}, there is a spectral sequence
\[
E^2_{p,q} = H_p(\pi_q(X_{\bullet})) \implies \pi_{p+q}(|X_{\bullet}|), 
\]
from filtering by skeleta.  
The same proof yields an equivariant version of this spectral sequence. By Proposition~\ref{charofdibar}, there is a weak equivalence 
\[ \iota_{D_{2m}}^*\THR(A) \simeq |B_{\bullet} (N_{D_2}^{D_{2m}}A,N_{e}^{D_{2m}}\iota_e^*A,N_{\zeta D_2 \zeta^{-1}}^{D_{2m}}c_{\zeta}A)|,
\]
so the spectral sequence in this case will be of the form:
\[
E^2_{p,q} = H_p(\m{\pi}_q^{D_{2m}}( N_{D_2}^{D_{2m}}A \wedge N_{e}^{D_{2m}}\iota_e^*A^{\wedge \bullet} \wedge N_{\zeta D_2 \zeta^{-1}}^{D_{2m}}c_{\zeta}A)) \Rightarrow \m{\pi}^{D_{2m}}_{p+q}(\THR(A)), 
\]
The edge homomorphism of this spectral sequence is a map
\[
\m{\pi}^{D_{2m}}_{p}(\THR(A)) \to H_p(\m{\pi}_0^{D_{2m}}( N_{D_2}^{D_{2m}}A\wedge N_{e}^{D_{2m}}\iota_e^*A^{\wedge \bullet} \wedge N_{\zeta D_2\zeta^{-1}}^{D_{2m}}c_{\zeta}A)). 
\]
We can identify the right hand side as
\begin{align}\label{rhs of above}
	H_p\left(\m{\pi}^{D_{2m}}_0N_{D_2}^{D_{2m}}A \square (\m{\pi}^{D_{2m}}_0N_{e}^{D_{2m}}\iota_e^*A)^{\square \bullet} \square  \m{\pi}^{D_{2m}}_0(N_{\zeta D_2\zeta^{-1}}^{D_{2m}}c_{\zeta}A) \right), 
\end{align}
using the collapse of the K\"unneth spectral sequence~\cite{LM06} in degree 0. By Definition~\ref{def: norm in Mackey}, 
there are isomorphisms of $D_{2m}$-Mackey functors 
\begin{align*}
\upi_0^{D_{2m}}(N_e^{D_{2m}}\iota_e^*A)\cong & N_e^{D_{2m}}\iota_e^*\upi_0^{D_2}(A),&\\
\upi_0^{D_{2m}}(N_{\zeta D_2 \zeta^{-1}}^{D_{2m}}c_{\zeta}A)\cong & N_{\zeta D_2 \zeta^{-1}}^{D_{2m}}c_{\zeta}\upi_0^{D_2}(A),\\
\upi_0^{D_{2m}}(N_{D_2}^{D_{2m}}A)\cong &N_{D_2}^{D_{2m}}\upi_0^{D_2}(A). &
\end{align*} 
We can therefore identify \eqref{rhs of above}
as 
\[
H_p\left( N_{D_2}^{D_{2m}}\m{\pi}_0^{D_2}A \square (N_{e}^{D_{2m}}\iota_e^*\m{\pi}_0^{D_2}A)^{\square \bullet} \square N_{\zeta D_2\zeta^{-1}}^{D_{2m}}c_{\zeta}\m{\pi}_0^{D_2}A   \right ) \,,
\]
and hence the edge homomorphism gives a linearization map
\[
\m{\pi}^{D_{2m}}_{p}(\THR(A))  \longrightarrow \HR_p^{D_{2m}}(\m{\pi}_0^{D_2}A) \,.
\]
To prove the claim that this map is an isomorphism in degree zero, we note that the only contribution to \(t+s=0\) is
\[ E_{0,0}^2\cong N_{D_2}^{D_{2m}}\underline{\pi}_0^{D_{2}}A\square_{N_e^{D_{2m}}\iota_e^*\underline{\pi}_0^{D_{2}}A} N_{\zeta D_2\zeta^{-1}}^{D_{2m}}c_{\zeta}\underline{\pi}_0^{D_{2}}A \] 
concentrated in degree $s=t=0$. By Proposition~\ref{prop:HRBox} this is $\HR_0^{D_{2m}}(\underline{\pi}_0^{D_{2}}A).$
Since this is a first quadrant spectral sequence, we observe that 
\[E_{0,0}^2\cong E_{0,0}^{\infty}\cong  \underline{\pi}_{0}^{D_{2m}}\THR(A).\]
\end{proof}
 \begin{remark}
 In~\cite{LewisC23}, Chloe Lewis constructs a B\"okstedt spectral sequence for Real topological Hochschild homology, which computes the equivariant homology of $\THR(A)$. The $E_2$-term of this spectral sequence is described by Real Hochschild homology, further justifying that $\HR$ is the algebraic analogue of THR.  
 \end{remark}

\section{Witt vectors of rings with anti-involution}\label{sec: witt}
Hesselholt--Madsen~\cite{HM97} proved that for a commutative ring $A$, 
\[
	\m{\pi}_0^{\mu_{p^n}}(\THH(A))(\mu_{p^n}/\mu_{p^n}) \cong \W_{n+1}(A; p). 
\]
This was extended to associative rings in~\cite{Hes97}. 
Recall that topological Hochschild homology is a cyclotomic spectrum, which yields restriction maps
\[
R_n\co \THH(A)^{\mu_{p^{n}}} \to \THH(A)^{\mu_{p^{n-1}}}.
\] 
One can then define
$
\TR(A;p) = \lim_{n,R_n} \THH(A)^{\mu_{p^n}},
$
and it follows from the Hesselholt--Madsen result above that 
\[
\pi_0 \TR(A;p) \cong \W(A;p) \,,
\]
where $\W(A; p)$ denotes the $p$-typical Witt vectors of $A$. This was then extended by Hesselholt to non-commutative Witt vectors. From~\cite[Theorem~A]{Hes97}, for an associative ring $A$ there is an isomorphism 
\[ \TC_{-1}(A;p)\cong \W(A;p)_F.\]
Here $W(A;p)$ denotes the non-commutative $p$-typical Witt vectors of $A$, and 
\[
\W(A;p)_F=\coker \left (1-F\co \W(A;p)\longrightarrow \W(A;p) \right ) \,,
\] 
where $F$ is the Frobenius map. Analogously, one would like to have a notion of (non-commutative) Witt vectors for discrete $E_{\sigma}$-rings, such that for an $E_{\sigma}$-ring $A$, $\m{\pi}_0^{D_{2m}}\THR(A)$ is closely related to the Witt vectors of $\m{\pi}_0^{D_2}A$. In this section, we define such a notion of Witt vectors.

Real topological Hochschild homology also has restriction maps 
\[
    R_n \co \THR(A)^{\mu_{p^n}} \to \THR(A)^{\mu_{p^{n-1}}}.
\]
We want to understand the algebraic analogue of these restriction maps, which requires defining a Real cyclotomic structure on Real Hochschild homology. To do this, we first recall from~\cite{BGHL19} the definition of geometric fixed points for Mackey functors. Let $G$ be a finite group and let $\mA$ be the Burnside Mackey functor for the group $G$. For $N$ a normal subgroup of $G$, let $\cF[N]$ denote the family of subgroups of $G$ such that $N\not\subset H$.
\begin{defin}
Fix a finite group $G$ and let $N\le G$ be a normal subgroup. Let $E\cF[N](\mA)$ be the subMackey functor of the Burnside Mackey functor $\mA$ for $G$ generated by $\mA(G/H)$ for all subgroups $H$ such that $H$ does not contain $N$. Then define 
\[ \widetilde{E}\cF[N](\mA)=\mA/(E\cF[N](\mA)).\]
If $\mM$ is a $G$-Mackey functor and $N$ is a normal subgroup of $G$, then 
\begin{align*} 
	E\cF[N](\mM):=\mM\square E\cF[N](\mA),&\text{ and }\\
	\widetilde{E}\cF[N](\mM):=\mM\square \widetilde{E}\cF[N](\mA).&
\end{align*}
More generally, if $\mM_{\bullet}$ is a dg-$G$-Mackey functor we define 
\begin{align*}
	(E\cF[N](\mM_{\bullet}))_n:=E\cF[N](\mM_n). 
\end{align*}
\end{defin}
Note that the $G$-Mackey functor $\widetilde{E}\cF[N](\mA)$ has the property that 
\[ 
	\widetilde{E}\cF[N](\mA)(G/H) = \begin{cases}
			0 & N\not\subset  H\\
			\mA((G/N)/(N/H)) & N\subset H \\
	\end{cases}
\]
which is desired for isotropy separation; i.e. there is an exact sequence 
\[  E\cF[N](\mA)\to \mA\to \widetilde{E}\cF[N](\mA)\ \]
which models the isotropy separation sequence. 

We now recall the definition of geometric fixed points for Mackey functors, as in~\cite{BGHL19}. Let $\mM$ be a $D_{2m}$-Mackey functor. By~\cite[Proposition~5.8]{BGHL19}, we know $\widetilde{E}\cF[\mu_d](\mM)$ is in the image of $\pi_{d}^*$, the pullback functor from $D_{2m}/\mu_d$-Mackey functors to $D_{2m}$-Mackey functors. Consequently, we may produce a $D_{2m}/\mu_d\cong D_{2m/d}$-Mackey functor 
$(\pi_{d}^*)^{-1}(\widetilde{E}\cF[\mu_d](\mM)).$
\begin{defin}[{Definition~5.10~\cite{BGHL19}}]
Let $\mM$ be a $D_{2m}$-Mackey functor. We define the  $D_{2m/d}$-Mackey functor of $\Phi^{\mu_d}$-geometric fixed points to be 
\[ \Phi^{\mu_d}(\mM):=(\pi_{d}^*)^{-1}(\widetilde{E}\cF[\mu_d]\mM).\]
\end{defin}

\noindent We now provide Real Hochschild homology with a Real cyclotomic structure.
\begin{prop}\label{cyclotomic}
Given $D_2\subset D_{2m}\subset O(2)$, $\mu_{d}$ a normal subgroup of $D_{2m}$ where $d \mid m$ and $\mM$ a discrete $E_{\sigma}$-ring, there is a natural isomorphism 
\[ \Phi^{\mu_{d}} ( \HR_\bullet^{D_{2m}}(\m{M} ) )\cong \HR_\bullet^{D_{2m/d}}(\m{M})\]
of simplicial $D_{2m/d}$-Mackey functors and consequently an isomorphism 
\[ \Phi^{\mu_{d}}\left (\HR_*^{D_{2m}}(\mM)\right )\cong \HR_*^{D_{2m/d}}(\mM)\]
of $D_{2m/d}$-Mackey functors. 
\end{prop}

\begin{proof}
We apply $\Phi^{\mu_{d}}$ level-wise to the bar construction
 \[
(\HR^{D_{2m}}(\mM))_{\bullet} =   B_{\bullet}(N_{D_2}^{D_{2m}}\mM,N_{e}^{D_{2m}}\iota_e^*\mM, N_{\zeta_{2m} D_2 \zeta_{2m}^{-1}}^{D_{2m}}c_{\zeta_{2m}}\mM).
\]
By~\cite[Proposition~5.13]{BGHL19},  
$\Phi^{\mu_{d}}$ is strong symmetric monoidal so 
\begin{equation}\label{iso 1}
	\Phi^{\mu_{d}} \left ( (\HR^{D_{2m}}(\mM))_k \right ) \cong  \Phi^{\mu_{d}} \left (N_{D_2}^{D_{2m}}\mM \right ) \square \left ( \Phi^{\mu_{d}} \left ( N_e^{D_{2m}}\iota_e^*\mM \right )  \right ) ^{\square k} \square \Phi^{\mu_{d}} \left ( N_{\zeta_{2m} D_2 \zeta_{2m}^{-1}}^{D_{2m}}c_{\zeta_{2m}} \mM  \right ).
\end{equation} 
 The interaction of the geometric fixed points and the norm is described in~\cite[Theorem~5.15]{BGHL19}. It follows that  
\begin{equation}\label{iso 2}	
	\Phi^{\mu_{d}} \left ( \HR^{D_{2m}}(\mM)_k \right )\cong  N_{D_2}^{D_{2m/d}}\mM \square \left ( N_e^{D_{2m/d}}\iota_e^*\mM \right ) ^{\square k} \square  N_{\zeta_{2m/d} D_2\zeta_{2m/d}^{-1}}^{D_{2m/d}}c_{\zeta_{2m/d}}\mM
\end{equation}
where we use the fact that 
$D_{2m}/\mu_d\cong D_{2m/d}$
by an isomorphism sending $\zeta_{m}$ to $\zeta_{m/d}$  and $\tau$ to $\tau$. Similarly, 
\[( \zeta_{2m}D_2\zeta_{2m}\cdot \mu_d)/\mu_{d}\cong \zeta_{2m/d}D_2\zeta_{2m/d}.\]
 since $\zeta_{2m} D_2 \zeta_{2m}^{-1}=\langle \zeta_m\tau \rangle$ and the isomorphism sends 
$\zeta_{m}\tau$ to $\zeta_{m/d}\tau$.  It therefore suffices to check that the simplicial structure maps commute with the isomorphisms \eqref{iso 1} and \eqref{iso 2}. 
Consider the isomorphism 
\[\mu_{2m(\bullet+1)} \longrightarrow D_{2m}/D_2\amalg D_{2m}^{\amalg \bullet } \amalg D_{2m}/\zeta_{2m} D_2\zeta_{2m}^{-1}\]
from Lemma~\ref{lem: isomorphism of D2m sets}. We observe that  
these maps form an isomorphism of simplicial $D_{2m}$-sets 
(cf. Remark~\ref{rem:simplicialD2msets}). Applying $\mu_d$-orbits, we have 
\begin{align}\label{eq: isomorphism of d2md mackey}
\mu_{2m(\bullet+1)/d} \longrightarrow D_{2m/d}/D_2\amalg D_{2m/d}^{\amalg \bullet } \amalg D_{2m/d}/\zeta_{2m/d} D_2\zeta_{2m/d}^{-1}.
\end{align}
 
If $\mM$ is the restriction of a $D_{2m}$-Mackey functor then the isomorphism is simply induced by tensoring with the isomorphism \eqref{eq: isomorphism of d2md mackey} of simplicial $D_{2m/d}$-sets. To see this, note that given a finite group $G$ with normal subgroup $N\le G$, the geometric fixed points $\Phi^{N}$ of the norm $N^T$ of a $G$-set $T$ is given by taking the norm $N^{T/N}$ of the orbits $T/N$. The more general statement also holds since we described the isomorphism level-wise and the compatibility with the face and degeneracy maps can be described on each box product factor indexing by the isomorphism of simplicial $D_{2m/d}$-sets \eqref{eq: isomorphism of d2md mackey} and using the compatibility of that isomorphism 
with the face and degeneracy maps.
\end{proof}

\begin{defin}\label{fixed points of Mackey functor}
Given a normal subgroup $N$ in $G$ we define a functor
\[  (-)^{N} \co \Mak_{G}\to \Mak_{G/N}\] 
as the composite $\upi_0^{G/N}\big((H(-))^{N}\big)$. 
\end{defin}
Given a $D_{2m}$-Mackey functor $\mM$, the $D_2$-Mackey functor $\mM^{\mu_m}$ is the data 
\[ 
	\begin{tikzcd} \mM(D_{2m}/D_{2m})  \ar[r, bend left=20]  & \ar[l, bend left=20] \mM(D_{2m}/\mu_m) 
	\end{tikzcd}
\]
regarded as a $D_2$-Mackey functor with the action of the Weyl group $W_{D_2}(e)=D_2$ given by the action of the Weyl group $W_{D_{2m}}(\mu_m)=D_{2m}/\mu_m\cong D_2$. 

\begin{remark}
    This Mackey ``fixed points'' functor is the functor denoted \(q_\ast\) in~\cite{HMQ}, where it was shown to preserve Green and Tambara functors. 
\end{remark}

These fixed points in Mackey functors relate to categorical fixed points. 
\begin{prop}
    For a \(G\)-spectrum \(E\) and for all integers \(k\), we have
    \[
        \underline{\pi}_k^{G/N}\big(E^N\big)\cong \big(\underline{\pi}_k^G(E)\big).
    \]
\end{prop}

\begin{remark}
This gives an alternate characterization of the geometric fixed points  $\Phi^{\mu_d}\m{M}$ of a $D_{2m}$-Mackey functor $\mM$ for $d|m$ as 
\[ \Phi^{\mu_d}\m{M} \cong (\widetilde{E}\cF[\mu_d]\mM)^{\mu_d}\]
since it is clear in this case that there is a natural isomorphism
\[(\widetilde{E}\cF[\mu_d](\mM))^{\mu_d}\cong \left (\pi_{\mu_d}^*\right )^{-1}(\widetilde{E}\cF[\mu_d](\mM)).\] 
\end{remark}
\begin{construction}
Given a simplicial $D_{2p^k}$-Mackey functor $\mM_{\bullet}$ there is a natural map 
\[ \mM_{\bullet}\to \widetilde{E}\cF[\mu_p](\mM_{\bullet}) \] 
and then an induced natural map 
\[ \left (\left (\mM_{\bullet} \right )^{\mu_{p}}\right )^{\mu_{p^{k-1}}} \to \left (\left (\widetilde{E}\cF[\mu_p](\mM_{\bullet})\right )^{\mu_p} \right )^{\mu_{p^{k-1}}}_. \] 
Note that we can identify 
\[ \left (\left (\mM_{\bullet} \right )^{\mu_{p}}\right )^{\mu_{p^{k-1}}} =\left (\mM_{\bullet} \right )^{\mu_{p^k}}\]
by unraveling the definition. So, there is a natural transformation 
\[ R_k\co \left ( -\right )^{\mu_{p^k}} \longrightarrow (\Phi^{\mu_p}(-))_.^{\mu_{p^{k-1}}}\] 
\end{construction}
 We give this map the name $R_k$ because in our case of interest, where $\m{M}_{\bullet} = \HR_{\bullet}^{D_{2p^k}}\!(\m{\pi}_0^{D_2}A)$ for an $E_{\sigma}$-ring $A$, it is an algebraic analogue of the restriction map $R_k$ on $\THR(A)$. 
\begin{construction} 
If $\mM$ is a discrete $E_{\sigma}$-ring, it follows from Proposition~\ref{cyclotomic} that there is an isomorphism of $D_{2p^{k-1}}$-Mackey functors 
\[ \Phi^{\mu_{p}}\HR_n^{D_{2p^k}}\! (\mM ) \cong  \HR_n^{D_{2p^{k-1}}}\!(\mM). \]
The above construction therefore produces restriction maps 
\[ R_k \co  \left ( \HR_n^{D_{2p^k}}\! (\mM) \right )^{\mu_{p^{k}}} \to \left ( \HR_n^{D_{2p^{k-1}}}(\mM ) \right )^{\mu_{p^{k-1}}}, \]
which are maps of $D_2$-Mackey functors. The maps $R_k$ can be described explicitly on Lewis diagrams by 
\[
\begin{tikzcd}
   \HR_n^{D_{2p^k}} (\mM )(D_{2p^k}/D_{2p^k})  \ar[d, bend left=50]\ar[r,"R_k"] &  
   \HR_n^{D_{2p^{k-1}}} (\mM)(D_{2p^{k-1}}/D_{2p^{k-1}})
   \ar[d, bend left=50]\\
   	\HR_n^{D_{2p^k}} (\mM) (D_{2p^k}/\mu_{p^k})
   \ar[u, bend left=50] \ar[r,"R_k"] &  
   \HR_n^{D_{2p^{k-1}}}(\mM )(D_{2p^{k-1}}/\mu_{p^{k-1}}).
   \ar[u, bend left=50]
\end{tikzcd}
\]
\end{construction}

\begin{defin}\label{def: Witt vectors}
Given a discrete $E_{\sigma}$-ring $\mM$, we define the \emph{truncated $p$-typical Real Witt vectors} of $\mM$ by the formula
\[\mW_{k+1}(\mM;p)=\HR_0^{D_{2p^k}}\left (\mM  \right )^{\mu_{p^k}}\]
and the \emph{$p$-typical Real Witt vectors} of $\mM$ as
\[ \mW(\mM;p):= \underset{k,R_k}{\lim} \HR_0^{D_{2p^k}}\left (\mM  \right )^{\mu_{p^k}} \]
where the limit is computed in $D_2$-Mackey functors. This is functorial in $\mM$ by naturality of the restriction maps $R$ so we produce a functor
\[ \mW(-;p)\co  \Alg_{\sigma}(\Mak_{D_2})\longrightarrow \Mak_{D_2}.\]
\end{defin}
\begin{remark}
If $\mM$ is a $D_2$-Tambara functor, then by Lemma~\ref{lem: HR0 of Tambara is Tambara} and~\cite[Prop 5.16]{HMQ},  
$\HR_0^{D_{2p^n}}\left (\mM  \right )^{\mu_{p^n}}$ is 
a $D_{2}$-Tambara functor and we produce  
\[ \mW(-;p)\co  \Tamb_{D_2}\longrightarrow \Tamb_{D_2}.\]
\end{remark}
We now consider how the $p$-typical Real Witt vectors are related to Real topological Hochschild homology. Let $A$ be an $E_{\sigma}$-ring. Since the family $\cR$ 
does not contain $\mu_{p^k}$ and the family $\cF[\mu_p]$ does not contain $\mu_{p^k}$ for any $k\ge 1$, on $\mu_{p^k}$-fixed points, we do not need to distinguish between these two. 

Recall that Real topological restriction homology is defined as 
 \[ \TRR(A;p)=\underset{k,R_k}{\holim} \THR(A)^{\mu_{p^{k}}}\]
in the category $\Sp_{\cV}^{D_2}$ {\cite[Definition 3.6.]{Hog16}}.

\begin{thm}\label{witt thr}
Let $A$ be an $E_{\sigma}$-ring and $\underline{M}=\upi_0^{D_2}A$. There is an isomorphism of $D_2$-Mackey functors
\[\m{\pi}_0^{D_2}\TRR(A;p) \cong \mW(\m{M};p)\]
whenever $R^1\lim_{k}\m{\pi}_1^{D_2}\THR(A)^{\mu_{p^k}}=0.$
\end{thm}
\begin{proof}
We note that there is an isomorphism
\begin{align*} 
	\upi_0^{D_{2p^k}}\THR(A)\cong \HR_0^{D_{2p^k}}(\mM) 
\end{align*}
of $D_{2p^k}$-Mackey functors by 
Theorem~\ref{thm:linearization} and consequently a natural isomorphism of $D_2$-Mackey functors 
\begin{align} \label{eq:iso with Witt}
	\upi_0^{D_2}(\THR(A)^{\mu_{p^k}})\cong (\HR_0^{D_{2p^k}}(\underline{M}))^{\mu_{p^k}}. 
\end{align}
 \noindent By construction, the diagram 
\[ 
	\xymatrix{ 
		\upi_0^{D_2}\THR(-)^{\mu_{p^k}} \ar[r]^{R_k} \ar[d]^{\cong} &  \upi_0^{D_2}\THR(-)^{\mu_{p^k-1}}  \ar[d]^{\cong} \\ 
		\left (\HR_0^{D_{2p^k}}( \upi_0^{D_2}(-) ) \right)^{\mu_{p^k}}  \ar[r]^-{R_k} &  \left (\HR_0^{D_{2p^{k-1}}} ( \upi_0^{D_2}(-)) \right )^{\mu_{p^{k-1}}}
		}
\]
commutes. To see this, we note that the edge homomorphism in the spectral sequence associated to the skeletal filtration of a simplicial spectrum (cf. Proof of~\ref{thm:linearization}) is compatible with the restriction maps $R_k$. 
Consequently, 
\begin{align*}
\m{\pi}_0^{D_2}\TRR(A;p) \cong & \lim_R \upi_0^{D_2}\THR(A)^{\mu_{p^k}} \\
\cong & \lim_R \left (\HR_0^{D_{2p^k}}(\upi_0^{D_2}A) \right )^{\mu_{p^k}}\\
= & \mW(\m{\pi}_0^{D_2}A;p)
\end{align*}
where the first isomorphism holds by our assumption that 
\[R^1\lim_{k}\upi_{1}^{D_2}\THR(A)^{\mu_{p^k}} =0.\]
\end{proof}
\begin{construction}\label{Frobenius}
We now define a Frobenius map 
\[F\colon \thinspace \mW(\mM;p)\to \mW(\mM;p).\]
The restriction maps 
$\text{res}_{\mu_{p^{k-1}}}^{\mu_{p^k}}$
and $\text{res}_{D_{2p^{k-1}}}^{D_{2p^k}}$  on the Mackey functor $\HR_0^{D_{2p^{k}}}\left (\mM  \right )$ induce a map of $D_{2}$-Mackey functors
\[
\begin{tikzcd}
    \left ( \HR_0^{D_{2p^{k}}}\left (\mM  \right )\right )^{\mu_{p^k}}(D_2/D_2) \ar[d, bend left=50]\ar[r] &  \left ( \HR_0^{D_{2p^{k}}}\left (\mM  \right )\right )^{\mu_{p^{k-1}}}(D_{2p}/D_{2})  \ar[d, bend left=50]\\
    \left ( \HR_0^{D_{2p^{k}}}\left (\mM  \right )\right )^{\mu_{p^k}}(D_2/e) \ar[u, bend left=50] \ar[r] &  \left ( \HR_0^{D_{2p^{k}}}\left (\mM  \right )\right )^{\mu_{p^{k-1}}}(D_{2p}/e)\ar[u, bend left=50]
\end{tikzcd}
\]
 that we call $F$. Observe that the target of this map can be identified with the $D_2$-Mackey functor  $\left (\HR_0^{D_{2p^{k-1}}}(\mM)\right )^{\mu_{p^{k-1}}}$, so we simply write 
\[ F_k\colon \thinspace   \left ( \HR_0^{D_{2p^{k}}}\left (\mM  \right )\right )^{\mu_{p^k}}\longrightarrow \left (\HR_0^{D_{2p^{k-1}}}(\mM)\right )^{\mu_{p^{k-1}}}\]
for this map. Note that this is natural so we can apply it to the map 
\[ \HR_0^{D_{2p^k}}(M)\longrightarrow \widetilde{E}\cF[\mu_p]\left ( \HR_0^{D_{2p^k}}(M)\right ).\]
Therefore, by construction and Proposition~\ref{cyclotomic}
this map is compatible with the restriction maps in the sense that there are commutative diagrams 
\[
\begin{tikzcd}
   \left ( \HR_0^{D_{2p^{k}}}\left (\mM  \right )\right )^{\mu_{p^k}} \ar[r,"F_k"] \ar[d, "R_k"] &  \left ( \HR_0^{D_{2p^{k-1}}}\left (\mM  \right )\right )^{\mu_{p^{k-1}}} \ar[d,"R_k"] \\
    \left ( \HR_0^{D_{2p^{k-1}}}\left (\mM  \right )\right )^{\mu_{p^{k-1}}} \ar[r,"F_k"] & \left (\HR_0^{D_{2p^{k-2}}}\left (\mM  \right )\right )^{\mu_{p^{k-2}}}
\end{tikzcd}
\]
Therefore, we have an induced map 
\[F_k\co \mW(\mM;p)\longrightarrow \mW(\mM;p).\]
\end{construction}
\begin{rem}\label{Vershiebung}
We can also define a Verschiebung operator 
\[ V\colon \thinspace \mW(\mM;p)\longrightarrow \mW(\mM;p)\]
in exactly the same way as in Construction~\ref{Frobenius} by replacing the restriction maps in the Mackey functor with the transfer maps in the Mackey functor. 
\end{rem}

There are also topological analogues of the maps $F$, $V$ and $R$ on $\THR(A)^{\mu_{p^k}}$ when $A$ is an $E_{\sigma}$-ring, which satisfy certain relations (cf.~\cite[\S 3]{Hog16}). In particular, $R_k$ and $F_k$ are compatible in the sense that $R_{k-1}\circ F_k=F_{k-1}\circ R_k$.
The cokernel of the map $\id_{\mW(A;p)}-F$ is defined to be the coinvariants $\mW(A;p)_F$. As a consequence, we have the following refinement of~\cite[Theorem A]{Hes97}. 
\begin{thm}\label{TCR thm}
Let $A$ be an $E_{\sigma}$-ring, and suppose that
\[
R^1\lim_{k}\m{\pi}_{1}^{D_2}\THR(A)^{\mu_{p^k}}=0.
\]
Then there is an isomorphism 
\[ \upi_{-1}\TCR(A;p)\cong \mW(\upi_0^{D_2}(A);p)_{F}.\]
\end{thm}
\begin{proof}
We compute the homotopy fiber of the topological map $\id-F$ by the long exact sequence in homotopy groups. By Theorem~\ref{witt thr}, we can identify $\pi_0\TRR(A)$. By inspection, the topological map $F\colon \thinspace \TRR(A)\to \TRR(A)$
induces the algebraic map
$F\co \mW(\upi_0A;p)\longrightarrow \mW(\upi_0A;p)$
and therefore $\m{\pi}_{-1}\TCR(A;p)$ is the cokernel of the algebraic map $\id-F$.
\end{proof}

\section{Computations}\label{sec: computations}
In this section, we use the new algebraic framework from Section~\ref{sec:HR} to do some concrete calculations. In particular, we compute the $D_{2m}$-Mackey functor
$\underline{\pi}_0^{D_{2m}}\THR(H\mZ)$
where $m\geq 1$ is an odd integer and $\m{\mathbb{Z}}$ is the constant Mackey functor. This computation has essentially already appeared in~\cite[Theorem 3.15]{DMP22}, but we include it to illustrate how our approach can be done entirely in the setting of homological algebra for Mackey functors and is therefore, in a sense, algorithmic. The first step in this computation is a Tambara reciprocity formula for sums, which we present for a general finite group and may be of independent interest. 
\subsection{The Tambara reciprocity formulae} 
The most difficult relations in Tambara functors tend to be the interchange describing how to write the norm of a transfer as a transfer of norms of restrictions. These are described by the condition that if 
\[
\begin{tikzcd}
    {U}
        \ar[d, "g"]
        &
    {T}
        \ar[l, "h"']
        &
    {U\times_S \prod_g(T)}
        \ar[l, "{f'}"']
        \ar[d, "{g'}"]
        \\
    {S}
        &
        &
    {\prod_g(T)}
        \ar[ll,"{h'}"']
\end{tikzcd}
\]
is an exponential diagram, then we have
\[
    N_g\circ T_h=T_{h'}\circ N_{g'}\circ R_{f'}.
\]
The formulae called ``Tambara reciprocity'' unpack this in two basic cases: \(g: G/H\to G/K\) is a map of orbits and
\begin{enumerate}
    \item \(h\colon G/H\amalg G/H\to G/H\) is the fold map or
    \item \(h\colon G/J\to G/H\) is a map of orbits.
\end{enumerate}
These respectively describe universal formulae for
\[
    N_H^K(a+b)\text{ and } N_H^K tr_J^H(a).
\]
In general, these can be tricky to specify, since we have to understand the general form of the dependent product (or equivalently here, coinduction). 

\begin{lem}[{\cite[Proposition 2.3]{HillMaz19}}]
    If \(h\colon T\to G/H\) is a morphism of finite \(G\)-sets, with \(T_0=h^{-1}(eH)\) the corresponding finite \(H\)-set, and if \(g\colon G/H\to G/K\) is the quotient map corresponding to an inclusion \(H\subset K\), then we have an isomorphism of \(G\)-sets
    \[
        \prod_{g}(T)\cong G\times_K \Map^H(K,T_0).
    \]
\end{lem}

This entire argument is induced up from \(K\) to \(G\), so it suffices to study the case \(K=G\). In this formulation, the map \(f'\) along which we restrict is the map
\[
G/H\times \Map^H(G,T_0)\to G\times_H T_0
\]
given by
\[
(gH,F)\mapsto \big[g,F(g)\big].
\]

Since the transfer along the fold map is the sum, we can understand the exponential diagram by further pulling back along the inclusions of orbits in \(\Map^H(G,T)\). Let \(F\in \Map^H(G,T_0)\) be an element, let \(G\cdot F\) be the orbit, and let \(K=\Stab(F)\) be the stabilizer. We can now unpack the orbit decomposition of \(G/H\times G/K\) and the maps to \(T\) and to \(G/K\cong G\cdot F\). We depict the exponential diagram, together with the pullback along the inclusion of the orbit \(G\cdot F\) and orbit decompositions of the relevant pieces in Figure~\ref{fig:ExtendedExponential}.

\begin{figure}[ht]
\includegraphics[angle=90, origin=c, height=.8\textwidth]{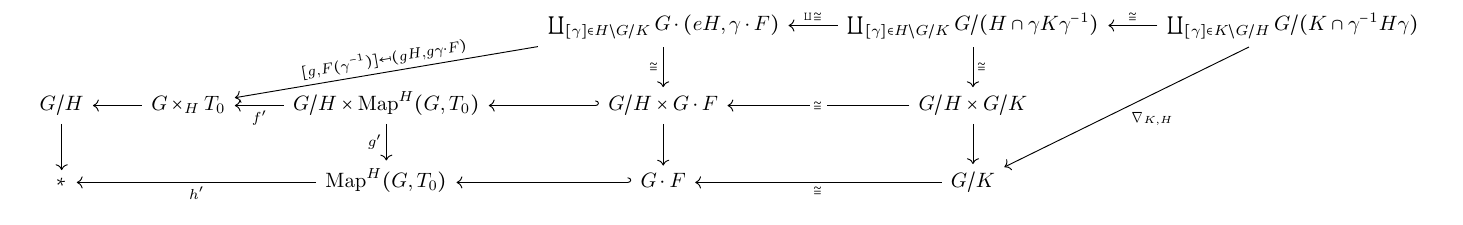}
\caption{Unpacking the exponential diagram on orbits}
\label{fig:ExtendedExponential}
\end{figure}

The map labeled \(\nabla_{K,H}\) is the coproduct of the canonical projection maps
\[
G/(K\cap\gamma^{-1}H\gamma)\to G/K,
\]
and the norm along this is, by definition
\[
N_{\nabla_{K,H}}=\prod_{[\gamma]\in K\backslash G/H} N_{K\cap \gamma^{-1}H\gamma}^{K}.
\]
Also by definition, the restriction along \(\prod c_{\gamma}\) is
\[
\prod_{[\gamma]\in H\backslash G/K} \underline{M}\big(G/(H\cap \gamma K\gamma^{-1})\big)
\xrightarrow{\prod_{[\gamma]\in K\backslash G/H} \gamma}
\prod_{[\gamma]\in K\backslash G/H} \underline{M}\big(G/(K\cap \gamma^{-1} H\gamma)\big)
\]
 for a Mackey or Tambara functor \(\underline{M}\), where \(\gamma\) here is the Weyl action. 
These give all the tools needed to understand the Tambara reciprocity formulae. We spell out the formula for a norm of a sum in general; we will not need the formula for the norm of a transfer here.

\begin{thm}\label{thm:TambRecipSums}
     Let \(G\) be a finite group and \(H\) a subgroup, and let \(\m{R}\) be a \(G\)-Tambara functor. For each \(F\in\Map^H\big(G,\{a,b\}\big)\), let \(K_F\) be the stabilizer of \(F\). Then for any \(a,b\in\m{R}(G/H)\), we have
  \[
    N_H^G(a+b)= \sum_{[F]\in\Map^H(G,\{a,b\})/G} tr_{K_F}^G\left(
    \prod_{[\gamma]\in K_F\backslash G\slash H}\!\! N_{K_F\cap (\gamma^{-1}H\gamma)}^{K_F}\Big(\gamma res_{H\cap (\gamma K\gamma^{-1})}^{H}\big(F(\gamma^{-1})\big)\Big)
    \right)
    \]
\end{thm}
\begin{proof}
    This follows immediately from the proceeding discussion. The only step to check is the identification of the restriction. This follows from the identification of the map \(f'\) with the evaluation map. In the case \(T_0=\{a,b\}\), where here we blur the distinction between \(a\) and \(b\) as elements of \(\m{R}(G/H)\) and as dummy variables, the map
    \[
    G/\big(H\cap (\gamma K\gamma^{-1})\big)\to G/H\times\{a,b\}
    \]
    coincides with the canonical quotient onto the summand specified by evaluating \(F\) at \(\gamma^{-1}\).
\end{proof}
Here we only need the Tambara reciprocity formulae for dihedral groups, so we now restrict attention to these cases.

\subsection{Formulae for dihedral groups}
We use the following two lemmas to describe the coinductions needed for the Tambara reciprocity formulae for
\[
    N_{D_{2m}}^{D_{2n}}(a+b)\text{ and }
    N_{D_{2m}}^{D_{2n}} tr_{\mu_m}^{D_{2m}}(a),
\]
where \(n/m\) is an odd prime.
\begin{lem}\label{unstable multiplicative double coset formula}
Let $H,K$ be subgroups of a finite group $G$ and $T$ be a $G$-set. Then there is a natural bijection 
\[ 	
\Map^H(G,T)^K\cong  \prod_{K \gamma  H\in K\backslash G \slash H}\Map^H(K\gamma H,T)^K\cong
\prod_{K \gamma  H\in K\backslash G \slash H}T^{(\gamma^{-1}K\gamma\cap H)}.
\]
When $T$ has trivial action this simplifies to
\[ 	
\Map^H(G,T)^K\cong \prod_{K \gamma H\in K\backslash G \slash H}T.
\]
\end{lem}
\begin{proof}
Regarding $G$ as a $K\times H^{\op}$-space, there is an isomorphism 
\[ 
\iota_{K\times H^{\op}}^*G\cong \coprod_{\gamma}K\gamma H
\]
where $\gamma$ ranges over representatives for double cosets $K\gamma H$ in $K\backslash G \slash H$. The \(K\) fixed points of coinduction up from \(H\) to \(G\) are the same as the \(K\times H^{\op}\)-fixed points of just the set of maps out of \(G\). This gives a natural (in \(T\)) bijection
\[ 	
\Map^H(G,T)^K\cong \Map^H\left(\coprod_{K\backslash G \slash H} K\gamma H,T\right)^K \cong \prod_{K\backslash G \slash H}\Map^H(K\gamma H,T)^K
\]
as desired. 

To understand each individual factor, we use the quotient map \(\pi_K\colon K\times H\longrightarrow H\) to rewrite the fixed points:
\[
\big(\Map^H(K\gamma H,T)\big)^K\cong \Map^{K\times H}(K\gamma H,\pi_K^\ast T).
\]
Since by definition of the pullback, \(K\) acts trivially on \(T\), the map factors through the orbits \(K\backslash K\gamma H\), which is an \(H\)-orbit. The classical double coset formula identifies this with \(H/(H\cap \gamma^{-1}K\gamma)\), and the result follows.

The simplification follows from the action being trivial, and hence all points being fixed.
\end{proof}

\begin{lem}\label{lem: coinduced}
Let  $p$ be an odd prime. There are isomorphisms of $D_{2p}$-sets
\[
\Map^{D_{2}}(D_{2p},\{a,b\}) \cong \ast\amalg\ast\amalg \coprod_{i=1}^{2^{(p+1)/2}-2} D_{2p}/D_2 \amalg \coprod_{i=1}^{((2^{p-1}-1)/p)+1-2^{(p-1)/2}}D_{2p}.
\]
\end{lem}

\begin{proof}
We observe that the only fixed points with respect to $\mu_{p}$ and any subgroup of $D_{2p}$ containing $\mu_{p}$ are the constant maps \(f_a\) and \(f_b\) sending $D_{2p}$ to $a$ or $D_{2p}$ to $b$, respectively. This gives the first two summands. The $D_2$-fixed points are given by a product of $(p+1)/2=|D_2\backslash D_{2p}\slash D_2|$ copies of $\{a,b\}$ with an additional copy corresponding to the constant maps. Combining this information, we have $2^{(p+1)/2}-2$ copies of $D_{2p}/D_2$ each contributing one $D_2$-fixed point. The remaining summands must be given by copies of the $D_{2p}$-set $D_{2p}$ and examining the cardinality the number of copies of $D_{2p}$ must be $((2^{p-1}-1)/p)+1-2^{(p-1)/2}$. 
\end{proof}

\begin{remark}
    There is a geometric interpretation of this. The \(D_{2p}\)-set \(\Map^{D_2}(D_{2p},\{a,b\})\) can be thought of as the ways to label the vertices of the regular \(p\)-gon with labels \(a\) or \(b\). The stabilizer of a function is the collection of those rigid motions which preserve the labeling. Those with stabilizer \(D_2\) are the ones that are symmetric with respect to the reflection through some fixed vertex and passing through the center. These then depend only on the label at the chosen vertex, and then \(\tfrac{p-1}{2}\) labels for the the next vertices, moving either clockwise or counterclockwise from that vertex. Of these, there are two that are special: the two where everything has a fixed label.
\end{remark}

\begin{notation}\label{not:Cardinalities Of Fixed Points}
    For an odd prime \(p\), let
    \[
        c_p=2^{\frac{p-1}{2}}-1\text{ and }
        d_p=\frac{2^{p-1}-1}{p}-c_p.
    \]
\end{notation}

\begin{lem}
Let $p$ be an odd prime. 
    We have an isomorphism of \(D_{2p}\)-sets
    \[
        \Map^{D_2}(D_{2p},D_2)\cong D_{2p}/\mu_p\amalg \coprod^{d_p+c_p} D_{2p}.
    \]
\end{lem}
\begin{proof}
    Since \(D_2\) is a free \(D_2\)-set, there are no \(D_{2p}\)-fixed points, by the universal property of coinduction. On the other hand, \(D_2\) as a \(D_2\)-set is actually in the image of the restriction from \(D_{2p}\)-sets via the quotient map \(D_{2p}\to D_2\), and this is compatible with the inclusion of \(D_2\) into \(D_{2p}\). This allows us to rewrite our \(D_{2p}\)-set as
    \[
        \Map^{D_2}(D_{2p},D_2)\cong \Map(D_{2p}/D_2,D_{2p}/\mu_p).
    \]
    Since \(\mu_p\) acts trivially in the target, the \(\mu_p\)-fixed points are 
    \[
        \Map(D_{2p}/D_2,D_{2p}/\mu_p)^{\mu_p}=\Map(D_{2p}/\mu_pD_2,D_{2p}/\mu_p)\cong D_{2p}/\mu_p.
    \]
    Finally, since \(D_2\)-acts freely in the target, there are no \(D_2\)-fixed points (and hence for any of the conjugates). Counting gives the desired answer.
\end{proof}

We need two much more general versions of these identifications, both of which follow from the preceding lemmas.

\begin{lem}\label{lem:NormalSG}
If \(N\subset H\subset G\) with \(N\) a normal subgroup of \(G\), then for any \(H\)-set \(T\) with \(T=T^N\), we have a natural bijection of \(G\)-sets
    \[
        \Map^H(G,T)\cong \Map^{H/N}(G/N,T).
    \]
\end{lem}

\begin{proof}
    Since \(N\) is a normal subgroup of \(G\), for any \(g\in G\), \(n\in N\), and \(f\in\Map^H(G,T)\), we have
    \[
        f(gn)=f(c_g(n)g)=c_g(n) f(g)=f(g),
    \]
    where the first equality is by normality of \(N\), the second is by \(H\)-equivariance of \(f\), and the third is by the condition that \(T^N=T\). 
\end{proof}

\begin{cor}
Let $p$ be an odd prime. For any \(m\geq 1\), there are isomorphisms of $D_{2pm}$-sets
    \[
        \Map^{D_{2m}}(D_{2pm},\{a,b\}) \cong \{f_a,f_b\}\amalg \coprod_{i=1}^{2c_p} D_{2pm}/D_{2m} \amalg \coprod_{i=1}^{d_p}D_{2pm}/\mu_{m}.
    \]
    and
    \[
        \Map^{D_{2m}}(D_{2pm},D_{2m}/\mu_m)\cong D_{2pm}/\mu_{pm}\amalg \coprod^{d_p+c_p} D_{2pm}/\mu_{m}.
    \]
\end{cor}
\begin{proof}
    This follows from Lemma~\ref{lem:NormalSG}. The subgroup \(N=\mu_{m}\). The quotient \(D_{2m}/\mu_m\) is \(D_2\); the quotient \(D_{2pm}/\mu_m\) is \(D_{2p}\), and the result follows from the previous lemmas.
\end{proof}

We will now produce a Tambara reciprocity formula for sums for the group $D_{2p}$ when $p$ is an odd prime.

\begin{notation}\label{set of words}
Let
\[
    X=\Big(\Map^{D_2}\big(D_{2p},\{a,b\}\big)\Big)^{D_{2}}-\{f_a,f_b\}
\]
be a set of representatives for the \(D_{2}\)-fixed points. 
For a point \(\underline{x}\in X\), let \(x_i=\underline{x}(\zeta_p^{i})\).
Let 
\[
    Y=\Big(\Map^{D_2}\big(D_{2p},\{a,b\}\big)-D_{2p}\cdot X\Big)/D_{2p}
\]
be the set of free orbits in \(\Map^{D_2}\big(D_{2p},\{a,b\}\big)\), 
and for an equivalence class \([\underline{y}]\in Y\), let \(y_i=\underline{y}(\zeta_p^i)\).
\end{notation}

\begin{lem}[Tambara reciprocity for sums for dihedral groups] \label{lem:reciprocity}
Let $D_{2p}$ be the dihedral group where $p$ is an odd prime, with a generator $\tau$ of order $2$ and $\zeta_p$ of order $p$, and let $D_2$ be the cyclic subgroup generated by $\tau$. 
Let $\m{S}$ be a $D_{2p}$-Tambara functor. Then for all $a$ and $b$ in $\m{S}(D_{2p}/D_2)$ 
\begin{align*}
    N_{D_2}^{D_{2p}}(a_{D_2}+b_{D_2}) = & N_{D_2}^{D_{2p}}(a) +N_{D_2}^{D_{2p}}(b) \\ 
    & + \sum_{\underline{x}\in X }\tr_{D_2}^{D_{2p}}\left (x_0\prod^{(p-1)/2}_{i=1}  N_{e}^ {D_{2}}(\zeta^{i}_p\res^{D_{2}}_{e}(x_i))\right )\\ 
    & +\sum_{[\underline{y}] \in Y}\tr_{e}^{D_{2p}}\left (\prod_{i=1}^p \zeta^{i}_p \res^{D_{2}}_{e}(y_i) \right )
\end{align*}
where  $X$ and $Y$ are as in Notation~\ref{set of words}.
\end{lem}
\begin{proof}
This follows from Theorem~\ref{thm:TambRecipSums}, using the identification of coinduction given by Lemma~\ref{lem: coinduced}.
\end{proof}
\begin{exm}
Explicitly, in the case of $p=3$, we have the formula 
\begin{align*}
N_{D_2}^{D_{6}}(a+b)=&N_{D_2}^{D_{6}}(a_{D_2})+N_{D_2}^{D_{6}}(b_{D_2})\\
&+\tr_{D_2}^{D_{6}}(b_{D_2} \cdot N_e^{D_2}(\zeta_5\cdot \res^{D_2}_e(a_{D_2})))\\
&+ \tr_{D_2}^{D_{6}}(a_{D_2} \cdot N_e^{D_2}(\zeta_5\cdot \res^{D_2}_e(b_{D_2}))).
\end{align*}
because in this case the set $Y$ is empty.
When $p=7$, abbreviating $a_e=\res_e^{D_2}a$ and $b_e=\res_e^{D_2}b$ there is a summand 
\[ \tr_e^{D_{14}}\left ( \xi_7b_e\cdot \xi_7^2b_e\cdot \xi_7^3a_e\cdot \xi_7^4b_e\cdot \xi_7^5a_e\cdot \xi_7^6a_e\cdot \xi_7^7a_e\right ).\]
\end{exm}
\subsection{Truncated \texorpdfstring{$p$}{p}-typical Real Witt vectors of  \texorpdfstring{$\m{\mathbb{Z}}$}{the constant Mackey functor Z}}
For $p$ an odd prime, we compute the $D_{2p^{k}}$-Tambara functor  $\m{\pi}_0^{D_{2p^k}}\THR(H\m{\mathbb{Z}})$,
using the formula 
\[ \m{\pi}_0^{D_{2p^k}}\THR(H\m{\mathbb{Z}}) \cong N_{D_2}^{D_{2p^k}}\mZ\square_{N_{e}^{D_{2p^k}}\iota_e^*\mZ} N_{\zeta D_2\zeta^{-1}}^{D_{2p^k}}c_{\zeta}\mZ \] 
from Theorem~\ref{thm:linearization} and Proposition~\ref{prop:HRBox}. 
We therefore begin by computing the Mackey functor norm, $N_{D_2}^{D_{2p^k}}\mZ$. 

 Since we will be working both with dihedral groups as groups and with them as representatives of isomorphism classes of \(D_{2m}\)-sets in the corresponding Burnside ring, we will use distinct notation to keep track.

\begin{notation}
    If \(T\) is a finite \(G\)-set, then let \([T]\) denote the isomorphism class of \(T\) as an element of the Burnside ring. When \(T=G/G\), we will also simply write this as \(1\).
\end{notation}

We also need some notation for generation of a Mackey functor, especially representable ones.

\begin{notation}
    If \(T\) is a finite \(G\)-set, let \(\mA_{T}\cdot f\) be \(\mA_T\), with the canonical element \(T\xleftarrow{=}T\xrightarrow{=}T\) named \(f\).
\end{notation}

\begin{lem}\label{z2todpknormofz}
Let $p$ be an odd prime. There is an isomorphism of $D_{2p}$-Tambara functors 
\[ 
N_{D_2}^{D_{2p}}(\mZ)\cong \mA^{D_{2p}}/\big(2-[D_{2p}/\mu_p]\big).
\]
\end{lem} 
\begin{proof}
The constant Mackey functor \(\mZ\) for \(D_2\) is the quotient of \(\mA^{D_2}\) by the element \(2-[D_2]\in\mA^{D_2}(D_2/D_2)\) using the conventions of Section~\ref{rep functors}. Equivalently, we can rewrite this as a coequalizer of maps, both of which are represented by multiplication by a fixed \(D_2\)-set:
\[
\begin{tikzcd}
\mA^{D_2}\ar[r,"2" ',shift right=1ex]\ar[r,shift left=1ex,"D_2"]& \mA^{D_2} .
\end{tikzcd}
\]
Extend this to a reflexive coequalizer by formally putting in the zeroth degeneracy. This represents \(\mZ\) as a sifted colimit of free Mackey functors:
\[
\begin{tikzcd}
    \mA^{D_2}\cdot a\oplus\mA^{D_2}\cdot b
        \ar[r, shift right=1ex, "d_0"']
        \ar[r, shift left=1ex, "d_1"] 
    &
    \mA^{D_2}
        \ar[l, "s_0" description]
        \ar[r]
    &
    \mZ,
\end{tikzcd}
\]
where \(s_0(1)=b\), and where
\[
    d_0(b)=d_1(b)=1\text{ and }
    d_i(a)=\begin{cases}
        [D_2] & i=0 \\
        2 & i=1.
    \end{cases}
\]
\noindent The norm commutes with sifted colimits, so we deduce that 
we have a reflexive coequalizer diagram
\[
\begin{tikzcd}[column sep=large]
    N_{D_2}^{D_{2p}}(\mA^{D_2}\cdot a\oplus \mA^{D_2}\cdot b)
        \ar[r,shift right=2.5ex,"N(d_1)" '] 
        \ar[r,shift left=2.5ex,"N(d_0)"] & 
    N_{D_2}^{D_{2p}}(\mA^{D_2})
        \ar[l,"N(s_0)" description]
        \ar[r]
        &
    N_{D_2}^{D_{2p}}(\mZ).
\end{tikzcd}
\]
\noindent The norm is defined by the left Kan extension of coinduction, so we have a canonical isomorphism for representable functors:
\[
 N_{D_2}^{D_{2p}} \big(\mA^{D_2}\oplus \mA^{D_2}\big)\cong N_{D_2}^{D_{2p}}\big(\mA_{\{a,b\}}^{D_2}\big)\cong \mA_{\Map^{D_2}(D_{2p},\{a,b\})},
\]
and the norm of the Burnside Mackey functor for \(D_2\) is the Burnside Mackey functor for \(D_{2p}\). Here and from now on we simply write $\mA_T$ for $\mA_T^{D_{2p}}$. Lemma~\ref{lem: coinduced} determines the \(D_{2p}\)-set we see here: 
\[
\Map^{D_2}\big(D_{2p},\{a,b\}\big)=\{f_a\}\amalg \{f_b\}\amalg\coprod_{\m{x}\in X} D_{2p}/D_2\cdot \m{x}\amalg \coprod_{[y]\in Y} D_{2p}\cdot [y].
\]
This decomposition gives a decomposition of the representable:
\[
\mA_{\Map^{D_2}(D_{2p},\{a,b\})}\cong 
\mA\cdot f_{a}\oplus
\mA\cdot f_{b}\oplus
\bigoplus_{\m{x}\in X} \mA_{D_{2p/D_2}}\cdot \m{x}\oplus
\bigoplus_{[y]\in Y} \mA_{D_{2p}}\cdot y.
\]
\noindent Since the direct sum is the coproduct in Mackey functors, we can view each summand in the coequalizer as independently introducing a relation on \(\mA\). We can therefore 
work one summand at a time, keeping track of the added relations.
By the Yoneda Lemma, maps from a representable Mackey functor \(\mA_T\cdot f\) to \(\mA\) are in bijective correspondence with elements of \(\mA(T)\), and the bijection is given by evaluating a map of Mackey functors on the canonical element \(f\).  To determine these, we work directly, using the definition of the representables. 
For a general summand parameterized by the orbit of a function \(D_{2p}/D_2\to\{a,b\}\), the value of the corresponding face map is built out of the functions values at the points of \(D_{2p}/D_2\). The slogan here is that this is simply a ``decategorification'' of the Tambara reciprocity formula we already described. 

The first case is the constant functions. Here, we have
\[
    d_i(f_*)=N_{D_2}^{D_{2p}}\big(d_i(*)\big),
\]
for \(*=a,b\). Both \(d_0\) and \(d_1\) agree on \(b\) with value \(1\), so the summand \(\mA\cdot f_b\) contributes no relation. For the summand \(\mA\cdot f_a\), we use that the norms in the Burnside Tambara functor are given by coinduction:
\begin{align*}
    d_0(f_a)&=N_{D_2}^{D_{2p}}\big([D_2]\big)=[D_{2p}/\mu_p]+(d_p+c_p)[D_{2p}]\text{ and}\\
    d_1(f_a)&=N_{D_2}^{D_{2p}}(2)=
    2+2c_p[D_{2p}/D_2]+d_p [D_{2p}],
\end{align*}
where 
\[
c_p=2^{\frac{p-1}{2}}-1\text{ and }d_p=\frac{2^{p-1}-1}{p}-c_p
\]
are as defined in Notation~\ref{not:Cardinalities Of Fixed Points}.
Coequalizing these two maps 
introduces a relation
\[
2+2c_p[D_{2p}/D_2]+d_p[D_{2p}]-\big([D_{2p}/\mu_p]+(d_p+c_p)[D_{2p}]\big),
\]
which simplifies to
\[
    \big(2-[D_{2p}/\mu_p]\big)+c_p\big(2[D_{2p}/D_2]-[D_{2p}]\big)
\]
in \(\mA(D_{2p}/D_{2p}\).

The second case we consider is the easiest one: the summands parameterized by \(Y\). 
Maps from \(\mA_{D_{2p}}\cdot y\) to \(\mA\) are in bijection with elements of \(\mA(D_{2p}/e)=\mZ\). The explicit 
value is the corresponding summand from the 
Tambara reciprocity formula:
\[
d_i(y)=
\prod_{j=0}^{p} \res_{e}^{D_2}\big(d_i(y_j)\big),
\]
since the Weyl action on the underlying abelian group in the Burnside Mackey functor is trivial. Since \(d_0(b)=d_1(b)=1\) and since 
\[
\res_{e}^{D_2}\big(d_0(a)\big)=\res_e^{D_2}\big(d_1(a)\big)=2,
\]
both face maps always agree on these summands, with value given by
\[
    d_i(y)= 2^{|y^{-1}(a)|}.
\]

Finally, the trickiest summands are the ones parameterized by \(X\). 
Since we are mapping out of 
\[
\mA_{D_{2p}/D_{2}}\cong \Ind_{D_2}^{D_{2p}}\mA^{D_2},
\]
by the induction-restriction adjunction, it suffices to understand instead the restriction to \(D_2\) of the target. Here we use the multiplicative double coset formula: for a general \(D_2\)-Mackey functor \(\mM\), we have 
\[
i_{D_2}^{\ast}N_{D_2}^{D_{2p}}\mM\cong \mM\square\bigbox_{D_2\backslash D_{2p}\slash D_2-D_2eD_2} N_{e}^{D_2} i_e^{\ast}\mM.
\]
For the Burnside Mackey functor, there is a confusing collision: every Mackey functor in this expression is the Burnside Mackey functor, so we cannot distinguish between \(\mA^{D_2}\) as itself or as \(N_{e}^{D_2}\mathbb{Z}\). Writing things in terms of the actual norms of a generic Mackey functor, helps disambiguate. To a function \(\m{x}\) with stabilizer \(D_2\), we have the corresponding summand 
from the 
Tambara reciprocity formula:
\[
    d_i(\m{x})=d_i(x_0)\cdot\prod_{j=1}^{(p-1)/2} N_{e}^{D_2}\big(\res_{e}^{\ast}d_i(x_j)\big)
\]
where again, the triviality of the Weyl group allows us to ignore it. 
Note also that with the exception of \(x_0\), we actually only see the restriction of \(d_i(x_j)\). As we saw in the second case, the two face maps here always agree, with value \(1\) if \(x_j=b\) and with value \(2\) if \(x_j=a\).

If \(x_0=b\), then 
\(
    d_0(\m{x})=d_1(\m{x}),
\)
since the product factors always agreed.

If \(x_0=a\), then we have 
\[
    d_i(f)=d_i(a)\cdot \prod_{j=1}^{(p-1)/2} N_{e}^{D_2}\res_e^{\ast} d_i(x_j)= d_j(a)(2+[D_2])^k,
\]
where \(k\) is the number of \(j\) between \(1\) and \((p-1)/2\) such that \(x_j=1\).  The coequalizer therefore induces the relation
\[
    \big(2-[D_2]\big)\cdot \big(2+[D_2]\big)^k.
\]
Since these are multiples of the case \(i=0\), so we deduce that all of these summands contribute exactly one relation:
\[
2-[D_2]\in\mA(D_{2p}/D_2).
\]

Summarizing, we have that the norm \(N_{D_2}^{D_{2p}}\mZ\) is
\[
\mA/\Big(
\big(2[D_{2p}/D_{2p}]-[D_{2p}/\mu_p]\big) + c_p \big(2[D_{2p}/D_2]-[D_{2p}]\big), \big(2[D_{2}/D_2]-[D_2]\big)
\Big). 
\]
where to help the reader keep track of where in the Mackey functor the relations are born, we replace \(1\in\mA(G/H)\) with \(H/H\).

This simplifies in several ways, however. 
Since transfers in the Burnside Mackey functor are given by induction, 
\[
c_p\big(2[D_{2p}/D_{2}]-[D_{2p}]\big)=c_p tr_{D_2}^{D_{2p}} \big(2[D_{2}/D_{2}]-[D_2]\big),
\]
so we can remove this from the first relations with impunity, giving
\[
\mA/\Big(2[D_{2p}/D_{2p}]-[D_{2p}/\mu_p], 2[D_{2}/D_{2}]-[D_{2}]\Big)
.
\]
We also have
\[
res_{D_2}^{D_{2p}}\big(2[D_{2p}/D_{2p}]-[D_{2p}/\mu_p]\big)=2[D_2/D_2]-[D_2],
\]
so we can now drop the second relation. This yields
\[
N_{D_2}^{D_{2p}} \mZ\cong\mA/\big(2-[D_{2p}/\mu_p]\big).
\]
Since \(\mZ\) is a Tambara functor, \(N_{D_2}^{D_{2p}}\mZ\) is, and as a quotient of \(\mA\), it has a unique Tambara functor structure.
\end{proof}

This has a somewhat surprising consequence: the form of the norm is the same as what we started with, in that we are coequalizing two maps represented by \(G\)-sets of cardinality \(2\). 
Induction gives the following generalization.

\begin{thm}\label{thm:NormToDmofZ}
    For any odd integer \(m\geq 1\), we have an isomorphism of \(D_{2m}\)-Tambara functors
    \[
        N_{D_2}^{D_{2m}}\mZ \cong \mA^{D_{2m}}/(2-[D_{2m}/\mu_{m}]).
    \]
\end{thm}

\subsection*{Unpacking the norm}
We pause here to unpack this definition some, since the quotient of Mackey functors by a congruence relation might be less familiar than the abelian group case. 

\begin{defin}
    For any odd natural number \(m\), let 
    \[
        \m{R}_m=N_{D_2}^{D_{2m}}\mZ.
    \]
\end{defin}

\begin{lem}\label{lem: restriction of the norm of Z}
    For any \(k\) dividing \(m\), we have
    \[
        i_{D_{2k}}^\ast \m{R}_m\cong \m{R}_k,
    \]
    and
    \[
        i_{\mu_k}^\ast \m{R}_m\cong \mA^{\mu_k}.
    \]
\end{lem}
\begin{proof}
    Theorem~\ref{thm:NormToDmofZ} writes the norm $\m{R}_m$ as the coequalizer of
    \[
        \begin{tikzcd}
            {\mA^{D_{2m}}}
                \ar[r, shift left=.5em, "2"]
                \ar[r, shift right=.5em, "{[D_{2m}/\mu_m]}"']
                &
            {\mA^{D_{2m}}.}
        \end{tikzcd}
    \]
    Since the restriction functor on Mackey functors is exact, for any subgroup \(H\), the restriction of the norm is the coequalizer of \(2\) and 
    \[
        \res_{H}^{D_{2m}}[D_{2m}/\mu_m]=[i_H^\ast D_{2m}/\mu_m].
    \]
    When \(H=D_{2k}\), we have 
    \[
        i_{D_{2k}}^\ast D_{2m}/\mu_m=D_{2k}/\mu_k,
    \]
    since there is a single double coset and \(D_{2k}\cap\mu_m=\mu_k\). This gives the first part.
    
    When \(H=\mu_m\), we have  
    \[
        i_{\mu_m}^\ast D_{2m}/\mu_m=2\mu_m/\mu_m,
    \]
    since \(\mu_m\) is normal. This implies that the restriction to \(\mu_m\) is \(\mA^{\mu_m}\), and hence the restriction to \(\mu_k\) is \(\mA^{\mu_k}\).
\end{proof}

Since we are coequalizing two maps from the Burnside Mackey functor to itself, the value at \(D_{2m}/D_{2m}\) can also be readily computed. We need a small lemma about the products of certain \(D_{2m}\)-orbits.

\begin{prop}\label{prop: effect of two maps on basis}
    Let \(m\) be an odd natural number. Let \(H\subset D_{2m}\) be a subgroup, and let \(\ell=\gcd(m,|H|)\). Then we have
    \[
        D_{2m}/H\times D_{2m}/\mu_m\cong \begin{cases}
            D_{2m}/\mu_{\ell} & |H|\text{ even},\\
            D_{2m}/H\amalg D_{2m}/H & |H|\text{ odd}.
        \end{cases}
    \]
\end{prop}
\begin{proof}
    Since \(m\) is odd, any subgroup \(H\) is conjugate to either \(D_{2k}\) or \(\mu_k\), for \(k\) dividing \(m\), and the two cases are distinguished by the parity of the cardinality. Hence \(D_{2m}/H\) is isomorphic to either \(D_{2m}/D_{2k}\) or to \(D_{2m}/\mu_k\) in exactly the two cases in the statement. The result follows from the isomorphism
    \[
        D_{2m}/H\times D_{2m}/\mu_m\cong D_{2m}\timesover{H} i_H^\ast D_{2m}/\mu_m,
    \]
    and our earlier analysis of the restrictions.
\end{proof}

\begin{cor}
    For any odd \(m\), \(\m{R}_{m}(\ast)\) is a free abelian group:
    \[
        \m{R}_m(\ast)\cong\mathbb Z\big\{[D_{2m}/D_{2k}]\mid k\vert m\big\}.
    \]
    The image of \([D_{2m}/D_{2k}]\in\mA^{D_{2m}}(\ast)\) is \([D_{2m}/D_{2k}]\), while the image of \([D_{2m}/\mu_k]\in\mA^{D_{2m}}(\ast)\) is \(2[D_{2m}/D_{2k}]\).
\end{cor}
\begin{proof}
    Proposition~\ref{prop: effect of two maps on basis} describes the effect of the two maps on the standard basis for the Burnside ring. We see that 
    \[
        \big(2-[D_{2m}/\mu_m]\big)\cdot [D_{2m}/\mu_k]=0,
    \]
    while
    \[
        \big(2-[D_{2m}/\mu_m]\big)\cdot [D_{2m}/D_{2k}]=2[D_{2m}/D_{2k}]-[D_{2m}/\mu_k].
    \]
    This gives both the additive result and the images.
\end{proof}

Lemma~\ref{lem: restriction of the norm of Z} shows then that the same statement is essentially true for the values at dihedral subgroups.

\begin{cor}
    For any odd \(m\) and any \(k\) dividing \(m\), we have an isomorphism
    \[
        \m{R}_m(D_{2m}/D_{2k})\cong \mathbb Z\big\{ [D_{2k}/D_{2j}]\mid j\vert k\big\}.
    \]
\end{cor}

We can also spell out the restriction and transfer maps here. The restriction and transfer to the odd order cyclic subgroups is easier, since there is a unique maximal one.

\begin{prop}
    The restriction map
    \[
        \m{R}_m(*)\to \m{R}_m(D_{2m}/\mu_m)\cong \mA^{\mu_m}(\mu_m/\mu_m)
    \]
    is given by
    \[
        [D_{2m}/D_{2k}]\mapsto [\mu_m/\mu_k].
    \]
    The transfer map is given by
    \[
        [\mu_m/\mu_k]\mapsto 2[D_{2m}/D_{2k}].
    \]
\end{prop}
\begin{proof}
    These follow from the restriction and induction in \(D_{2m}\)-sets, together with the relation
    \[
        [D_{2m}/\mu_k]=2[D_{2m}/D_{2k}]
    \]
    in $\m{R}_m$.
\end{proof}

For the restrictions and transfers to dihedral subgroups, we consider a maximal proper divisor. Let \(p\) be a prime dividing \(m\), and let \(k=m/p\).

\begin{prop}\label{prop: res and tr}
    The restriction map
    \[
        \m{R}_m(*)\to \m{R}_k(*)
    \]
    is given by
    \[
        [D_{2m}/D_{2j}]\mapsto \frac{p\ell}{j}[D_{2k}/D_{2\ell}],
    \]
    where \(\ell=\gcd(k,j)\). The transfer maps are given by
    \[
        [D_{2k}/D_{2j}]\mapsto [D_{2m}/D_{2j}].
    \]
\end{prop}
\begin{proof}
    The transfer maps are immediate. For the restriction, since \(m\) is odd, the normalizer of any dihedral subgroup is itself. The intersection of \(D_{2k}\) with \(D_{2j}\) is the dihedral group \(D_{2\ell}\), while the intersection of \(D_{2k}\) with any conjugate of \(D_{2j}\) is just the intersection \(\mu_j\cap\mu_k=\mu_\ell\). This means that it suffices to 
    count cardinalities. This gives
    \[
        i_{D_{2k}}^{\ast} D_{2m}/D_{2j}=D_{2k}/D_{2\ell}\amalg \coprod^{a} D_{2k}/\mu_\ell,
    \]
    where \(a=\frac{p\ell-j}{2j}\).
    Since \([D_{2k}/\mu_{\ell}]=2[D_{2k}/D_{2\ell}]\), the result follows.
\end{proof}

\begin{thm}\label{propz}
Let $m\ge 1$ be an odd integer. There is an isomorphism of $D_{2m}$-Mackey functors
\[ \underline{\pi}_0^{D_{2m}}\THR(H\mZ)\cong \mA^{D_{2m}}/(2-[D_{2m}/\mu_{m}]).\]
where $(2-[D_{2m}/\mu_m])$ is the ideal generated by $2-[D_{2m}/\mu_m]$ in the Tambara functor $\mA^{D_{2m}}$. 
\end{thm}
\begin{proof}
By Proposition~\ref{prop:HRBox} and Theorem~\ref{thm:linearization}, it suffices to compute the coequalizer 
\[ 
	\xymatrix{ 
		N_{D_2}^{D_{2m}}\mZ \square N_{e}^{D_{2m}}\mathbb{Z} \square N_{\zeta D_2\zeta^{-1}}^{D_{2m}} c_{\zeta}\mZ  \ar@<.5ex>[r] \ar@<-.5ex>[r] & N_{D_2}^{D_{2m}}\mZ\square N_{\zeta D_2\zeta^{-1}}^{D_{2m}} c_{\zeta}\mZ.
		}
\]
For any \(G\), \(N_e^G\mathbb{Z}\) is the Burnside Mackey functor, the symmetric monoidal unit. The \(E_0\)-structure map here is just the unit 
\[
\mA^{D_{2m}}\to N_{D_2}^{D_{2m}}\mZ.
\]
By Theorem~\ref{thm:NormToDmofZ}, this is surjective, so \(N_{D_2}^{D_{2m}}\mZ\) is a ``solid'' Green functor in the sense that the multiplication map is an isomorphism. Finally, note that the argument we gave to identify \(N_{D_2}^{D_{2m}}\mZ\) did not depend on the choice of \(D_2\) inside \(D_{2m}\), so we have an isomorphism
\[
N_{D_2}^{D_{2m}} \mZ\cong N_{\zeta D_2\zeta^{-1}}^{D_{2m}} \mZ
\]
of Tambara functors. We deduce that all pieces in the coequalizer diagram are just \(N_{D_2}^{D_{2m}}\mZ\).
\end{proof}
When restricted to $\upi_0^{D_2}(\THR(H\m{\mathbb{Z}})^{D_{2p^k}})$, our computation recovers the computation in~\cite{DMP22}. 
 \begin{cor}Let $p$ be an odd prime. 
 Then there are isomorphisms of abelian groups
 \[ \m{\pi}_0^{D_{2p^k}}(\THR(H\mZ))(D_{2p^k}/D_{2p^k}) \cong \m{\pi}_0^{D_{2p^k}}(\THR(H\mZ))(D_{2p^k}/\mu_{p^k})\cong \W_{k+1}(\mathbb{Z};p).\]
 \end{cor}  
Since we computed the restriction and transfer maps, we have the following computation. 
\begin{cor}
There is an isomorphism of Mackey functors
\[ \mathbb{W}_{k}(\m{\mathbb{Z}};p)=\underline{\W_{k}(\mathbb{Z};p)}\]
for odd primes $p$.
\end{cor}

\bibliographystyle{plain}
\bibliography{THR-Witt}
\end{document}